%% file: relaxedhwmodules.tex
\documentclass[11pt,reqno,oneside]{amsart}
\input{preamble}

\begin{document}

\author{Claude Eicher}
\address{SWITZERLAND}
\email{claudeeicher@gmail.com}
\date{\today}
\setcounter{tocdepth}{3}
\maketitle

\begin{abstract}
The relaxed highest weight representations introduced by Feigin et al. are
a class of representations of the affine Kac-Moody algebra $\widehat{\mathfrak{sl}_2}$,
which do not have a highest (or lowest) weight. We formulate a generalization of
this notion for an arbitrary affine Kac-Moody algebra $\mathfrak{g}$ with respect to
an arbitrarily chosen simple root. We then realize induced $\mathfrak{g}$-modules
of this type and their duals as global sections of twisted $\sh{D}$-modules on the Kashiwara flag scheme
$X$ associated to $\mathfrak{g}$. The $\sh{D}$-modules that appear in our construction are
direct images from subschemes of $X$ given by the intersection of finite dimensional
Schubert cells with their translate by a simple reflection. They are objects of
the category $\Hol(\lambda)$ of twisted holonomic $\sh{D}$-modules on $X$ introduced by
Kashiwara-Tanisaki. Besides the twist $\lambda$, they depend on a complex number
describing the monodromy of the local systems we construct on these intersections. 
We describe the global sections of the $*$-direct images as a module over the Cartan subalgebra
of $\mathfrak{g}$ and show that the higher cohomology vanishes. We obtain a complete
description of the cohomology groups of the direct images as $\mathfrak{g}$-modules in the following two cases.  
First, we address the case when the intersection is isomorphic to $\Gm$. This case can be reduced to considerations on the projective line. 
Second, we address the case of the $*$-direct image from an arbitrary intersection when the twist is regular antidominant
and the monodromy is trivial. For the proof of this case we introduce an auto-equivalence of $\Hol(\lambda)$ induced by the
automorphism of $X$ defined by a lift of a simple reflection. 
These results describe for the first time explicit non-highest weight
$\mathfrak{g}$-modules as global sections on the Kashiwara flag scheme and extend several results of Kashiwara-Tanisaki to the case of relaxed highest weight representations.   
\end{abstract}

\tableofcontents

\section{Introduction and Organization}

\setstretch{1.2}

The main result of this article\footnote{This article is based on the author's PhD thesis \cite{Eic16} written
under the supervision of G. Felder at ETH Zurich.} is to establish a geometric realization of 
certain relaxed highest weight $\mathfrak{g}$-modules on the Kashiwara flag scheme
in the setup of the seminal work \cite{KT95}. Here $\mathfrak{g}$ is an 
affine Kac-Moody algebra. 
In \autoref{sec:repthofaffinekmas}-\autoref{sec:relatedresults} we provide a 
thematic introduction. The layout of the article is described in \autoref{sec:matching}.  

\subsection{Relaxed highest weight representations}\label{sec:repthofaffinekmas}
Let $\overset{\circ}{\mathfrak{g}}$ be a finite dimensional semisimple Lie algebra over $\C$.
Affine Kac-Moody (Lie) algebras \cite{Kac90} are central extensions of the \emph{loop Lie algebra}
$\overset{\circ}{\mathfrak{g}}\otimes_{\C}\C[t,t^{-1}]$ of $\overset{\circ}{\mathfrak{g}}$ (plus extension
by a derivation) and certain twisted versions of this construction. 
Similar to the case of $\overset{\circ}{\mathfrak{g}}$, a class of representations of $\mathfrak{g}$ considered in the literature are
the highest weight representations. The highest weight representations
satisfying some finiteness conditions form the \emph{category $\sh{O}$} of $\mathfrak{g}$
of Bernstein-Gelfand-Gelfand (see \cite{BGG76} for
a definition in the finite dimensional case and \cite{Kum02} in the Kac-Moody case). 
We mention in passing that \cite{Fre07} defines the much larger category of smooth $\mathfrak{g}$-representations
inspired by the same-named notion for representations of $\overset{\circ}{G}(\mathbb{F}_q((t)))$, where
$\overset{\circ}{G}$ is a linear algebraic group with Lie algebra $\overset{\circ}{\mathfrak{g}}$. 
However, a description of it as explicit as in the case of category $\mathcal{O}$ is not available. 
The works \cite{SS97}, \cite{FST98},
\cite{FSST98} introduce a new type of representations of the simplest affine Kac-Moody algebra $\widehat{\mathfrak{sl_2}}$ satisfying a 
so-called \emph{relaxed highest weight condition}, requiring the subalgebra $\mathfrak{sl}_2\otimes_{\C} t\C[t]$
to act locally nilpotently on the representation in conjunction with additional finiteness conditions. This is therefore only slightly weaker than the usual highest weight condition. The first relaxed highest weight representations to look at
are the Verma-type (induced)
representations called \emph{relaxed Verma modules}. Their structure was shown, by determining the so-called embedding
diagrams \cite{SS97}[Section 3], to be reminiscent of the usual Verma modules for $\widehat{\mathfrak{sl_2}}$, albeit the relaxed Verma modules depend on one more complex parameter.
\par
A possible motivation for the study of relaxed highest weight representations comes from two-dimensional
conformal field theory, precisely the \emph{(chiral) WZW model for $\SU(2)$ and $\SL(2,\R)$}. 
The classical case is the case of $\SU(2)$ and nonnegative integer level $k$, where the spectrum of the theory is given
by the finitely many integrable representations of $\widehat{\mathfrak{sl}_2}$ at this level. 
In the more sophisticated case of rational level, this model is formulated algebraically and was investigated in
\cite{Gab01}, \cite{Rid10}, \cite{Rid11}, \cite{Fje11}, \cite{RW15} (for specific values of the level).
Namely, one considers certain representations
of a \emph{vertex operator algebra} defined as the simple quotient of
the vacuum representation of $\widehat{\mathfrak{sl}_2}$ at level $k$. These representations
are in particular $\widehat{\mathfrak{sl}_2}$-representations. 
That they are of highest weight is not required from the outset. What is however
required is that the category of these representations closes under the so-called fusion
and conjugation operation. In this way the above works arrive at the
conclusion that it is natural to take relaxed highest weight representations into account. Indeed,
\cite{RW15}[p. 629] write ``[...] the relaxed category $\mathcal{R}$ is a rich source of modules for
affine Kac-Moody (super)algebras that appears to be even more relevant to conformal field theory than the much more familiar category $\mathcal{O}$.''
\par
\emph{Going from $\widehat{\mathfrak{sl}_2}$ to a general affine Kac-Moody algebra} $\mathfrak{g}$ 
it is natural to relax the highest weight condition on its representations for an \emph{arbitrarily chosen simple root} $i$ of $\mathfrak{g}$.
Then the corresponding root vector (in the positive part of $\mathfrak{g}$) may act not locally nilpotently on
such a representation. This is our reason for calling the $\mathfrak{g}$-modules appearing in this
article relaxed highest weight modules. Our contribution on the side of representation theory is to provide a set of definitions of such $\mathfrak{g}$-modules.  
In \cite{Eic15} we define a class of $\mathfrak{sl}_2$-modules by explicit relations and based on this we
define induced $\mathfrak{g}$-modules from the minimal parabolic subalgebra of $\mathfrak{g}$ associated to $i$ in \autoref{sec:gmodulesRVermalambdalpha(veesl2)}. 
The latter are denoted by $\RVerma(\lambda,\mu)$ and $\RVerma(\lambda,\mu)^{\veesl2}$ and
depend on an element $\lambda$ in the weight lattice of $\mathfrak{g}$ and on a complex parameter $\mu$. The $\RVerma(\lambda,\mu)$ are analogues of the relaxed Verma modules of \cite{FST98}. 
These definitions suffice to describe the cohomology groups of sheaves we construct as $\mathfrak{g}$-modules.\par
All the above cited works deal with explicit constructions of relaxed highest weight
representations, but the \emph{categories} formed by relaxed highest weight representations await to be systematically explored
as it has been achieved for category $\sh{O}$ of $\mathfrak{g}$.  
Indeed, these categories are expected to have a similarly explicit description as category $\sh{O}$,
but their combinatorics should be slightly more involved, and this is one of the reasons why we advocate the
study of these categories. Some of the concrete tasks in this study
are to parametrize the simple objects, check the existence of a duality functor, determine when a relaxed Verma module
is irreducible, determine the Hom-spaces between relaxed Verma modules and the composition series of 
relaxed Verma modules. To try to solve them with the help of algebraic geometry is certainly a standard
approach and the results of this article can provide first steps towards this goal.

\subsection{Matching representation theory with geometry}\label{sec:matching}
Let us now describe our geometric realization of relaxed highest weight representations
on the Kashiwara flag scheme associated to an affine Kac-Moody algebra. Such a realization has been given for highest weight 
representations in the seminal work \cite{KT95}, which is the main
reference for this article.
It translates questions in representation theory to questions concerning sheaves of $\sh{D}$-modules,
 this translation working at the categorical level and not just for single objects.
An important difference
to the classical finite dimensional case is the absence of the Beilinson-Bernstein theorem \cite{BB81} (the subsequent discussion will however reveal that partial results are still valid in the affine case). We recall the setup of \cite{KT95} in \autoref{sec:setup}, in particular the notion of $\sh{D}$-modules on the Kashiwara flag scheme,
and introduce notation related to affine Kac-Moody algebras. 
\cite{KT95} consider twisted $\sh{D}$-modules constructed as direct images from the 
(finite dimensional) \emph{Schubert cells} $X_w$ in the Kashiwara flag scheme $X$ associated
to $\mathfrak{g}$. These are affine spaces indexed by the elements $w$ of the Weyl group $W$ of $\mathfrak{g}$. The basic reason why the Schubert cells enter the picture is
the fact that the positive part $\mathfrak{n}$ of $\mathfrak{g}$ acts locally nilpotently on highest
weight representations, which are to be constructed as global sections of these $\sh{D}$-modules. 
The twist is an arbitrary integral weight $\lambda$ of $\mathfrak{g}$, which determines a line bundle on $X$
and is a basic parameter in what follows. 
$\sh{D}$-module theory provides a $*$-direct image $\sh{B}_w(\lambda)$ from the subscheme $X_w$ as well as its holonomic dual, the $!$-direct image $\sh{M}_w(\lambda)$. 
The construction of these direct images crucially depends on the fact that the $X_w$ are affinely embedded into
 smooth ambient varieties.  
There is also a $!*$-direct image $\sh{L}_w(\lambda)$, which is by construction a simple object. 
$\sh{B}_w(\lambda)$ and $\sh{M}_w(\lambda)$ belong to the Serre subcategory
$\Hol_0(\lambda)$ generated by the $\sh{L}_v(\lambda)$, $v \in W$, inside the \emph{category $\Hol(\lambda)$
of holonomic $\lambda$-twisted (right) $\sh{D}$-modules on $X$} with support in some $\overline{X_w}$
introduced in \cite{KT95}. 
\par
We instead study in \autoref{ssec:intersection} the intersection $X_w \cap s_i X_w$ for $w$ such that $s_i w < w$,
where $s_i$ denotes the simple reflection associated to $i$ and $<$ denotes the Bruhat order.  It is isomorphic to $\Gm \times \Aff^{\length(w)-1}$ and the explicit form of the isomorphism given in Lemma \autoref{Lemma:intersectionofXwandXsiW} is relevant to our arguments.
In fact,  in the same way as the $X_w$ arise as orbits for the Iwahori group scheme $\Loop^+ I$ on the \emph{thin flag variety} $X^{\thin}=\Loop \overset{\circ}{G}/\Loop^+ I$ of \cite{BD91},
these intersections of Schubert cells arise
as \emph{orbits for the subgroup $\Loop^+ I \cap \dot{s_i}(\Loop^+ I)\dot{s_i}^{-1}$ of the Iwahori group scheme $\Loop^+ I$}. Here $\dot{s_i} \in \Loop \overset{\circ}{G}$ is a lift of $s_i$
and we assume $\mathfrak{g}$ untwisted. 
Thus, this stratification refines the stratification of $X^{\thin}$ by Schubert cells.
In the simplest case of $\mathfrak{g}=\widehat{\mathfrak{sl}_2}$ and $i=1$ this
orbit stratification can be illustrated by the diagram
\vskip 1mm
\begin{align}\label{eq:diagram}
\xymatrix{X_1 & \ar[l] X_{s_0} & \ar[l] \ar[dl] X_{s_0s_1} & \ar[l] \ar[dl] X_{s_0s_1s_0} \\
& \ar[dl] \ar[ul] X_{s_1}\cap s_1 X_{s_1}  & \ar[ul] \ar[l] \ar[dl] X_{s_1s_0}\cap s_1X_{s_1s_0} & \ar[ul] \ar[l] \ar[dl]
X_{s_1s_0s_1}\cap s_1X_{s_1s_0s_1} & \dots \\
s_1 X_1 & \ar[l] s_1 X_{s_0} & \ar[ul] \ar[l]  s_1 X_{s_0s_1} & \ar[l] \ar[ul] s_1 X_{s_0s_1s_0}} 
\end{align}
\vskip 3mm

continuing infinitely to the right. 
In this diagram, the dimension of the orbits is $0,1,2,\dots$ from left to right and there is a sequence of arrows connecting
the orbit $A$ with the orbit $B$ if and only if $B \subseteq \partial A$ (with respect to the Zariski topology). We note that $s_1$ is obviously a symmetry of this diagram,
interchanging the orbit $X_w$ with $s_1 X_w$ and inducing an automorphism of $X_w \cap s_1 X_w$. The $\Loop^+I$-orbit $X_w$ splits into the two $\Loop^+ I \cap \dot{s_1}(\Loop^+ I)\dot{s_1}^{-1}$-orbits
$X_w \cap s_1 X_w$ and $s_1 X_{s_1 w}$ for $s_1 w < w$, while it is also a $\Loop^+ I \cap \dot{s_1}(\Loop^+ I)\dot{s_1}^{-1}$-orbit in
the complementary case. 
\par

We then describe the cohomology of the $\sh{D}$-modules
constructed as direct images from the subscheme $X_w \cap s_i X_w$, where again \emph{$i$ is arbitrary but fixed}. The above description of $X_w \cap s_i X_w$ as an orbit makes us expect that the Lie subalgebra
$\mathfrak{n} \cap s_i \mathfrak{n} = \mathfrak{n}_i$ of $\mathfrak{n}$, which consists
of all positive root spaces of $\mathfrak{g}$ except the one associated to the simple root $i$,
 will act locally nilpotently on these cohomology groups. They should thus yield relaxed highest
weight representations of $\mathfrak{g}$ in the sense of \autoref{sec:repthofaffinekmas} and their
explicit description as $\mathfrak{g}$-modules in \autoref{sec:identificationofHjR!*wlambdaalphaasgmoduleinsomecases}
 confirms this in interesting cases. 
In contrast to the Schubert cell $X_w$, the intersection $X_w \cap s_i X_w \cong \Gm \times \Aff^{\length(w)-1}$
is \emph{not simply connected}. The $\Gm$-factor supports nontrivial local systems, which we consider in their $\sh{O}$-module version
(since we are working in the $\sh{D}$-module setting). Namely, we consider the local
systems of rank one with monodromy $e^{2\pi i \mu}$, where $\mu$ is a complex parameter; this is
defined relative to a choice of coordinate on $\Gm$ (also, parameters differing by an integer amount to isomorphic
local systems). The holonomic dual of such a local system has parameter $-\mu$. 
These local systems on $\Gm$ yield \emph{local systems on $X_w \cap s_i X_w$}, still parametrized by
$\mu$, of which we then
form the $*$-direct image $\sh{R}_{*w}(\lambda,\mu)$ and its holonomic dual, the $!$-direct image
$\sh{R}_{!w}(\lambda,\mu)$, in \autoref{sec:constructionsforXwandsXw}. These
objects of $\Hol(\lambda)$ with support $\overline{X_w}$ can thus be compared with the
objects $\sh{B}_w(\lambda)$ and $\sh{M}_w(\lambda)$ of \cite{KT95}. Here it is crucial that $X_w \cap s_i X_w$, like
the Schubert cell, is affinely embedded into the smooth ambient varieties. 
We first describe the structure of the global sections of $\sh{R}_{*w}(\lambda,\mu)$
as a module over the Cartan subalgebra of $\mathfrak{g}$ in 
\autoref{sec:H0ofR*wlambdaalphaashmodule} and show in \autoref{sec:cohomologyvanishing} that $\sh{R}_{*w}(\lambda,\mu)$ has no higher cohomology. 
Later we prove in Corollary \autoref{Cor:cohvanishingforR!w(lambda)}(1) that the higher cohomology 
of $\sh{R}_{!w}(\lambda,\mu)$ vanishes for $\mu \in \Z$, general $w$ and $\lambda$ regular antidominant. 
\par

We introduce in \autoref{sec:exactautoequivwidetildesi*ofHol(lambda)} 
a functor $\widetilde{s_i}_*$. We need it in the course of the proof of one of the main theorems of this article, Theorem \autoref{Thm:Ho0R*wlambdaasgmod}, but in addition it has a nice geometric meaning. 
The functor $\widetilde{s_i}_*$ can, as the notation suggests, be thought of as
a $\sh{D}$-module direct image functor induced by the automorphism of $X$ given by the
action of the generator $\widetilde{s_i}$ of the \emph{Tits extension of $W$}.
$\widetilde{s_i}_*$ defines an auto-equivalence
of the category $\Hol(\lambda)$. The effect of $\widetilde{s_i}_*$ can already be seen on the
level of the support of the $\sh{D}$-module $\sh{M} \in \Hol(\lambda)$: $\supp \sh{M} \subseteq \overline{X_v}$ is translated to  $\supp(\widetilde{s_i}_*\sh{M}) \subseteq s_i\overline{X_v}$. This in particular shows that $\widetilde{s_i}_*$ does not preserve the subcategory $\Hol_0(\lambda)$ and that $\widetilde{s_i}_*$ realizes the symmetry of the diagram \eqref{eq:diagram} at the level of $\sh{D}$-modules. 
The $\mathfrak{g}$-action on the cohomology groups of $\widetilde{s_i}_*\sh{M}$ is
the twist $(\cdot)^{\widetilde{s_i}}$ by the automorphism of $\mathfrak{g}$ induced by $\widetilde{s_i}$ of the
$\mathfrak{g}$-action on the cohomology groups of $\sh{M}$ as we prove in Theorem \autoref{Thm:Hojofs_*M}. 
In accord with the fact that $\widetilde{s_i}_*$ does not preserve $\Hol_0(\lambda)$, this twist of a highest weight representation of $\mathfrak{g}$ is in general, e.g. in the
case of a Verma module, no longer of highest weight. 
It is also natural to introduce, with the help of the functor $\widetilde{s_i}_*$, 
the Serre subcategories $\Hol_0^{\RVerma}(\lambda)$ and $\Hol_0^{\RVerma \Z}(\lambda)$
of $\Hol(\lambda)$ similar to $\Hol_0(\lambda)$ and containing the latter. We
do this in \autoref{ssec:categoryHol0R(lambda)}.\par

In \autoref{sec:identificationofHjR!*wlambdaalphaasgmoduleinsomecases}
we identify the $\mathfrak{g}$-module structure of the cohomology of 
$\sh{R}_{*w}(\lambda,\mu)$ and $\sh{R}_{!w}(\lambda,\mu)$ in interesting cases. 
This includes a relaxed version of parts of the main theorem \cite{KT95}[Theorem 3.4.1]
(this theorem requires $\lambda$ to be regular antidominant). 
We describe in \autoref{ssec:thecasew=si} the cohomology of $\sh{R}_{*w}(\lambda,\mu)$
and $\sh{R}_{!w}(\lambda,\mu)$ in case $w=s_i$ and for all values of $\lambda$ and $\mu$. 
This description includes the identification of the parameter $\mu$, which is of geometric origin, with the same-named parameter of the analogue of the relaxed Verma module mentioned in \autoref{sec:repthofaffinekmas}. 
Based on the fact that $X_{s_i}\cap s_i X_{s_i}$, isomorphic to $\Gm$, sits inside the
Schubert curve $\overline{X_{s_i}}$, which is isomorphic to $\Proj^1$ and closed in $X$, this description 
can be reduced to the description of the cohomology of twisted $\sh{D}$-modules on $\Proj^1$
arising as extensions from $\Gm$. The solution of the latter task is the main result of \cite{Eic15}. 
In this way, the analysis of \autoref{ssec:thecasew=si} can be thought of as an elaborate version of
the identification of the cohomology of the skyscraper $\sh{D}$-module $\sh{B}_1(\lambda)=\sh{M}_1(\lambda)$, where the distinguished point $1$ in $X$ is replaced by the Schubert curve $\overline{X_{s_i}}$. We identify in Theorem \autoref{Thm:Ho0R*wlambdaasgmod} the global sections of $\sh{R}_{*w}(\lambda,\mu)$ for $w \in W$ general, $\mu \in \Z$ and $\lambda$ regular antidominant as a $\mathfrak{g}$-module
belonging to the class introduced in \autoref{sec:gmodulesRVermalambdalpha(veesl2)}. The proof is based on a (geometrically constructed) short exact sequence
\begin{align*}
0 \rightarrow \widetilde{s_i}_*\sh{B}_w(\lambda) \rightarrow \sh{R}_{*w}(\lambda,\mu) \rightarrow \sh{B}_{s_iw}(\lambda) \rightarrow 0
\end{align*}
in $\Hol(\lambda)$. Applying the above mentioned Theorem \autoref{Thm:Hojofs_*M} and \cite{KT95}[Theorem 3.4.1(i) and (ii)] yields injective $\mathfrak{g}$-maps from $\MVerma(s_i w \cdot \lambda)$ and its $\widetilde{s_i}$-twist 
$\MVerma(s_iw\cdot \lambda)^{\widetilde{s_i}}$ to the $\mathfrak{g}$-module dual to the global sections of $\sh{R}_{*w}(\lambda,\mu)$. 
Here $\cdot$ denotes a shifted action of $W$ on the weight lattice. An algebraic
statement (Proposition \autoref{Lemma:nonzerosubofRwdotlabmda=intersects}) about the $\mathfrak{g}$-modules defined in \autoref{sec:gmodulesRVermalambdalpha(veesl2)} then allows the
identification of the global sections as such a $\mathfrak{g}$-module. 

\subsection{Related results}\label{sec:relatedresults}
When $\mathfrak{g}$ is  instead a finite dimensional semisimple Lie algebra and the Schubert cells $X_w$ are considered, results analogous to
the identification of global sections as $\mathfrak{g}$-modules in \cite{KT95}[Theorem 3.4.1] appear in
\cite{Kem76}[Theorem 7], \cite{BK81}[equation (5.1.2) and Corollary 5.8] and \cite{Gai05}[Theorem 10.6], while the vanishing of
higher cohomology is part of the Beilinson-Bernstein theorem. 
Let us come back to the case of affine Kac-Moody algebras. 
While \cite{KT95}[Theorem 3.4.1] is related to $\mathfrak{g}$ at the negative level since $\lambda$ was assumed
regular antidominant, it is important to mention that its
positive level analogue exists in the form of \cite{Kas900}[Theorem 5.2.1]. There the twist $\lambda$ is required 
to be \emph{dominant} and direct images from the \emph{finite codimensional Schubert cells}  $X^w$ in the Kashiwara flag scheme $X$ are considered. 
It should be possible to carry over the constructions of this article to this case. 
In the case of the finite dimensional Schubert cells
$X_w$ in $X^{\thin}$, the more recent work \cite{FG09} gives the following reformulations. 
\cite{FG09}[Theorem 5.2] states a cohomology vanishing result following
\cite{BD91}, see also \cite{Sha10}[Proposition 5.3.1].  In \cite{FG09}[Theorem 5.5] an \emph{equivalence}
between the category of Iwahori-equivariant twisted $\sh{D}$-modules on $X^{\thin}$ (more precisely, its monodromic
version) and a \emph{block} of (a variant of) category $\sh{O}$ of $\mathfrak{g}$ is shown, combining the results of \cite{BD91} with those of \cite{KT95}. It seems reasonable to try to extend this equivalence to an equivalence between a category of $\Loop^+ I \cap \dot{s_i}\Loop^+ I\dot{s_i}^{-1}$-equivariant twisted $\sh{D}$-modules 
on $X^{\thin}$ and a category of relaxed highest weight modules
and the results of this article might find an application in such an endeavor. 

\subsection{Notation}
We indicate by a superscript in brackets, e.g. $\mathfrak{n}^{(-)}$, that both variants are considered
separately, with and without a superscript, in the example $\mathfrak{n}$ and $\mathfrak{n}^-$.
The comma is used similarly, see e.g Definition \autoref{Def:DmodR*!wlambdaalphal}. 
The term variety and scheme means algebraic variety and scheme over $\C$ respectively. 
Notation used in the main text:\\
\vskip 2mm
\begin{tabular}{ll}
$\sqcup$ & disjoint union\\
$\backslash$ & difference of sets or quotient by a group acting on the left\\
$e,h,f$ & standard basis of the Lie algebra $\mathfrak{sl}_2$\\
$\Uea L$ & universal enveloping algebra of the Lie algebra $L$\\
$\Sym V$ & symmetric algebra of the vector space $V$ \\
$V^*$ & dual of a finite dimensional vector space $V$\\
$\Lie G$ & Lie algebra of a linear algebraic group $G$ over $\C$\\
$\C_{\mu}$ & One dimensional $\mathfrak{h}$-module defined by $\mu \in \mathfrak{h}^*$\\
$L-\modulecat$ & category of left modules over a Lie algebra $L$\\
$\Aff^n$ & affine $n$-space\\
$\Gm$ & multiplicative group\\
$\sh{D}_X$ & sheaf of differential operators on a smooth variety $X$\\
$\Omega_X$ & right $\sh{D}_X$-module of top-degree forms on $X$\\
$f_{\cdot}, f^{-1}$ & direct and inverse image functor for sheaves of sets associated to\\
& a morphism $f$ of varieties\\
$f^*\sh{L}$ & pullback of the line bundle $\sh{L}$ by $f$\\
$\supp \sh{F}$ & support of the sheaf $\sh{F}$
\end{tabular} 
\vskip 2mm

\subsection{Acknowledgements}
I am indebted to my supervisor G. Felder for his attention to my work,
for letting me present the results of this article and for his comments on
the manuscript. I would also like to thank A.S. Khoroshkin and S. Arkhipov
for several important comments related to this work. 

\setstretch{1.4}

\section{Setup}\label{sec:setup}
We mostly follow \cite{KT95}[Section 1 and 2].

\subsection{Affine Kac-Moody algebras and associated group schemes}
 Let $(\mathfrak{h},(\alpha_i)_{i \in I}, (h_i)_{i \in I})$
be affine Kac-Moody data \cite{Kac90}. Here $I$ is a finite set. 
Let $\mathfrak{g}$ be the Kac-Moody Lie algebra constructed from this data. 
Let $\Phi \subseteq \mathfrak{h}^*$ be the set of roots of $\mathfrak{g}$. 
Let $\Phi^{>0}$ and $\Phi^{<0}=-\Phi^{>0}$ be the subset of positive and negative roots respectively.
Let $\mathfrak{g} = \mathfrak{n}^-\oplus \mathfrak{h}\oplus \mathfrak{n}$
be the triangular decomposition of $\mathfrak{g}$.
We fix a $\Z$-lattice $P \subseteq \mathfrak{h}^*$ such that
$\alpha_i \in P$ and $P(h_i) \subseteq \Z$ for all $i \in I$.
Let $i \in I$. We denote the corresponding simple generators by $e_i \in \mathfrak{n}$ and 
$f_i \in \mathfrak{n}^-$. We also define the
following Lie subalgebras of $\mathfrak{g}$. We 
set $\mathfrak{b}^{(-)}=\mathfrak{h}\oplus \mathfrak{n}^{(-)}$.
We set $\mathfrak{p}_i = \mathfrak{b}\oplus \C f_i$, 
the direct sum of the remaining root spaces is denoted by $\mathfrak{n}_i^-$ so that
we have $\mathfrak{g}=\mathfrak{p}_i \oplus \mathfrak{n}_i^-$. Similarly we define $\mathfrak{p}_i^-$ and $\mathfrak{n}_i$.
We also set $\mathfrak{g}_i = \C e_i \oplus \mathfrak{h} \oplus \C f_i$
and $\mathfrak{g}_i^{\prime} = \C e_i \oplus \C h_i \oplus \C f_i$. For $l \in \Z_{\geq 0}$ we denote by 
$\mathfrak{n}^-_l \subseteq \mathfrak{n}^-$ the Lie subalgebra consisting of root spaces associated to
the negative roots of $\mathfrak{g}$ of height $\leq -l$. 

\subsubsection{Weyl group} Let $W$ be the Weyl group of $\mathfrak{g}$ \cite{Kac90}.
For $i \in I$ we denote the reflection corresponding to $\alpha_i$ by $s_i \in W$. $W$ acts on $\mathfrak{h}^*$ and leaves $\Phi$ invariant. Let $\rho \in \mathfrak{h}^*$ be such that
$\rho(h_i)=1$ for all $i \in I$. We define the $\cdot$-action of $W$ on $\mathfrak{h}^*$
by $w\cdot \lambda = w(\lambda+\rho)-\rho$. This is independent of the choice of $\rho$. We denote by $<$ the Bruhat 
partial order on $W$ and by $\length: W \rightarrow \Z_{\geq 0}$ the length function.

\subsubsection{Pro-unipotent group schemes}
 For a subset $\Psi \subseteq \Phi^{>0}$ such that
$(\Psi+\Psi)\cap \Phi^{>0} \subseteq \Psi$ we define a pro-unipotent affine group scheme $U(\Psi) \subseteq U = U(\Phi^{>0})$
and similarly for $\Psi \subseteq \Phi^{<0}$ we define $U^-(\Psi) \subseteq U^-$ as in \cite{KT95}[Section 1]. 
If $l \in \Z_{\geq 0}$ and $\Phi^{<0}_l \subseteq \Phi^{<0}$ denotes the subset of roots of height $\leq -l$, then we
denote $U^-_l = U^-_l(\Phi^{<0}_l)$. 

\subsubsection{Group schemes $T$, $B^-$, $P_i^-$ and $G_i$}\label{ssec:groupschemesTBPiGi}
We define an algebraic torus $T$ by $T=\Spec \C[P]$, where $\C[P]$ is the group algebra of 
the abelian group $P$. For $i \in I$ there is an affine group scheme $P_i^{(-)}$ associated to $\mathfrak{p}_i^{(-)}$.
It contains the Borel group scheme $B^{(-)} = T \ltimes U^{(-)}$ associated to $\mathfrak{b}^{(-)}$ as a closed subgroup scheme. Furthermore, there is a reductive group $G_i$ such
that $P_i^- = G_i \ltimes U^-(\Phi^{<0}\setminus\{-\alpha_i\})$. The Lie algebra of $G_i$ is $\mathfrak{g}_i$.  

\subsection{Kashiwara flag scheme}\label{sec:Kashiwaraflagscheme}
Let $X$ denote the Kashiwara flag scheme of
\cite{Kas90} (indeed a scheme over $\C$)
associated to the data $\mathfrak{g}$, $\mathfrak{b}$ and $P$. It is constructed as $X=G/B$, the quotient of a scheme $G$ (a version
of the loop group) by the locally free right action of $B$. $P_i^-$ acts on $G$ and hence on $X$ for $i \in I$. 
The so-called Tits extension $\widetilde{W}$ of $W$ 
also acts on $G$ and hence on $X$, see also \cite{Eic16}[Chapter 3]. The image of a $T$-invariant subset $Y$ of $X$ under
the action of $\widetilde{w} \in \widetilde{W}$
only depends on the image of $\widetilde{w}$ in $W$.  Hence $w Y$ for $w \in W$ is defined independent 
of the choice of lift of $w$.
$X$ is covered by open subschemes $N(X_w)=wN(X_1) \cong \Aff^{\infty}$, $w \in W$. Usually $N(X_1)$ is referred to as the big cell of $X$. 
There are $T$-equivariant isomorphisms of schemes
\begin{align}\label{eq:NXw}
U^-(\Phi^{<0}\cap w\Phi^{<0})\times U(\Phi^{>0}\cap w\Phi^{<0}) \xrightarrow{\cong} N(X_w)\;, (u_1,u_2) \mapsto u_1u_2wB\;. 
\end{align}
The image of $1\times U(\Phi^{>0}\cap w\Phi^{<0})$ is the (finite dimensional) Schubert cell $X_w $ in $X$.
It is isomorphic to $\Aff^{\length(w)}$. 
We will abbreviate $U_w = U(\Phi^{>0}\cap w\Phi^{<0})$ and $U^-_w = U^-(\Phi^{<0}\cap w\Phi^{<0})$
in what follows. We also define $X^{\leq w} = \bigcup_{y \leq w} N(X_y)$. The $X^{\leq w}$, $w \in W$, form a
cover of $X$ by $B^-$-invariant quasi-compact open subschemes of $X$.
For each $\lambda \in P$ there is a natural line bundle $\sh{O}_X(\lambda)$ on $X$ that is $P_i^-$-equivariant for each $i$.  
By definition, the fiber of $\sh{O}_X(\lambda)$ at the distinguished point $1 \in X$ has a right
$B$-action given by $B \twoheadrightarrow T \xrightarrow{-\lambda} \Gm$.  

\subsection{$\sh{D}$-modules on the Kashiwara flag scheme}
We recall how \cite{KT95}[Section 2] give a meaning to twisted $\sh{D}$-modules on $X$ and how they define their
 cohomology groups; this is not
immediate as $X$ is not locally of finite type. 

\subsubsection{Twisted $\sh{D}$-modules}
When $Y$ is a smooth variety, $\sh{L}$ a line bundle on $Y$
and $Z$ a closed subset of $Y$ we denote by $\Hol\left(\sh{D}^{\sh{L}}_Y,Z\right)$  the 
category of holonomic right $\sh{D}^{\sh{L}}_Y$-modules with set-theoretic support in $Z$. Here $\sh{D}^{\sh{L}}_Y$
is the sheaf of $\sh{L}$-twisted differential operators on $Y$. This
category is abelian, each object has finite length and it possesses an auto-equivalence
called holonomic duality $\Dual$. The following proposition is well-known, see e.g. \cite{KT96}[Proposition 1.2.8]. 

\begin{Prop}  \label{Prop:equivalencebetweenholonomiccats}\noindent
\begin{enumerate}
\item Let $f: X_1 \rightarrow X_2$ be a morphism of smooth varieties. 
Let $Z_1$ be a closed subvariety of $X_1$ and $Z_2$ be a closed
subvariety of $X_2$ such that $f$ induces an isomorphism $Z_1 \xrightarrow{\cong} Z_2$. 
Then $f_*$ induces an equivalence $\Hol(\sh{D}_{X_1}^{f^*\sh{L}},Z_1) \simeq \Hol(\sh{D}_{X_2}^{\sh{L}},Z_2)$. 
\item Let $\iota_{V_1}: V_1 \hookrightarrow Z_1$ be a smooth locally closed subvariety. 
Then $\iota_{f(V_1)}: f(V_1) \hookrightarrow Z_2$ is also a smooth locally closed subvariety. 
$f_*$ sends $\iota_{V_1*}(\Omega_{V_1})\otimes f^*\sh{L} \in \Hol(\sh{D}_{X_1}^{f^*\sh{L}},Z_1)$ to $\iota_{f(V_1)*}(\Omega_{f(V_1)})\otimes \sh{L}
\in \Hol(\sh{D}_{X_2}^{\sh{L}},Z_2)$. 
\item $f_* \Dual = \Dual f_*$ as a functor $\Hol(\sh{D}_{X_1}^{f^*\sh{L}},Z_1)
\rightarrow \Hol(\sh{D}_{X_2}^{\sh{L}},Z_2)$
\end{enumerate}\qed
\end{Prop}

Coming back to the Kashiwara flag scheme $X$, the group scheme $U^-_l$ acts freely on $X^{\leq w}$ and $X^{\leq w}_l = U^-_l \backslash X^{\leq w}$ is a smooth quasi-projective variety for all $l \in \Z_{>0}$ large enough. 
This is a result of \cite{SVV14}[Lemma A.2]. 
\emph{The following convention is important: whenever considering $X^{\leq w}_l$ in the sequel it will be understood that $l$ is large enough.} 
The closed immersion $\overline{X_w} \hookrightarrow X^{\leq w}$ 
of the Schubert variety then induces
a closed immersion $\overline{X_w} \hookrightarrow X^{\leq w}_l$. We will denote the projections
by $p_l: X^{\leq w} \rightarrow X^{\leq w}_l$ and $p^k_l: X^{\leq w}_k \rightarrow X^{\leq w}_l$
for $k \geq l$. 
As an instance of the above category we consider
$\Hol\left(\sh{D}_{X^{\leq w}_l}(\lambda),\overline{X_w}\right)$. Here $\sh{D}_{X^{\leq w}_l}(\lambda)$ is the sheaf
of differential operators on $X^{\leq w}_l$ twisted by $\sh{O}_{X^{\leq w}_l}(\lambda)$, 
a line bundle on $X^{\leq w}_l$ satisfying 
\begin{align}\label{eq:pullback}
p_l^* \sh{O}_{X^{\leq w}_l}(\lambda) \xrightarrow{\cong}
\sh{O}_X(\lambda)\vert X^{\leq w}\quad p^{k*}_l\sh{O}_{X^{\leq w}_l}(\lambda) \xrightarrow{\cong} \sh{O}_{X^{\leq w}_k}(\lambda)\;. 
\end{align}
It will be convenient to work with right $\sh{D}$-modules in this article because forming direct images
is easier than for left $\sh{D}$-modules.
Furthermore one considers a 
category $\Hol(\lambda, \overline{X_w}, X^{\leq w},l_0)$ whose objects are systems
$(\sh{M}_l)_{l \geq l_0}$, where $\sh{M}_l$ is an object of $\Hol\left(\sh{D}_{X^{\leq w}_l}(\lambda),\overline{X_w}\right)$ and for different $l$ the $\sh{M}_l$ are related by isomorphisms
\begin{align}\label{eq:gammakl}
\gamma^k_l: p^k_{l*}\sh{M}_k \xrightarrow{\cong} \sh{M}_l
\end{align}
satisfying the coherence condition
\begin{align}\label{eq:gammaformtransitivesystem}
\gamma^{l_1}_{l_3} = \gamma^{l_2}_{l_3} p^{l_2}_{l_3*}\gamma^{l_1}_{l_2}\;,\; l_1 \geq l_2 \geq l_3\;.
\end{align}
Even though the $\gamma^k_l$ are part of the data, we will often omit them for brevity. 
We have for each $l$ an equivalence of categories
\begin{align}\label{eq:equivalenceHol}
\Hol(\lambda, \overline{X_w}, X^{\leq w},l_0) \rightarrow \Hol\left(\sh{D}_{X^{\leq w}_l}(\lambda),\overline{X_w}\right)\;, (\sh{M}_l)_{l \geq l_0}
\mapsto \sh{M}_l\;.
\end{align} 
The category $\Hol(\lambda)$ is then defined by taking limits, see \cite{KT95},
where it is denoted by $\mathbb{H}(\lambda)$.
$\Dual$ induces an auto-equivalence of $\Hol(\lambda)$ that will be 
denoted by the same symbol.\par

\subsubsection{Cohomology groups $\Ho^j(X,\cdot)$
and $\overline{\Ho^j}(X,\cdot)$}\label{ssec:cohomologygroups}
 Let $\sh{M}=(\sh{M}_l)_{l \geq l_0}$ be an object
of $\Hol(\lambda, \overline{X_w}, X^{\leq w},l_0)$. For each $l$ we have
the usual sheaf cohomology groups $\Ho^j(X^{\leq w}_l, \sh{M}_l)$, $j \geq 0$, carrying
a left action of $\mathfrak{b}^-$ and these form an inverse system when $l$ varies. 
Here the left $\mathfrak{b}^-$-action is defined so that it differs by a sign from the natural right $\mathfrak{b}^-$-action on
 $\Ho^j(X^{\leq w}_l, \sh{M}_l)$, see \cite{KT95}[p. 39].
One defines $\Ho^j(X,\sh{M})=\varprojlim_{l \geq l_0} \Ho^j(X^{\leq w}_l, \sh{M}_l)$. 
One furthermore constructs a left $\mathfrak{g}$-action on $\Ho^j(X,\sh{M})$ from the right $\sh{D}$-module structure on $\sh{M}_l$ 
as follows. Let $v \in \mathfrak{g}$.
Then there is an $m \geq 0$ such that $[v,\mathfrak{n}^-_{l+m}] \subseteq \mathfrak{n}^-_l$
for all $l \geq 0$ and $v$ defines a section
\begin{align}\label{eq:partiallplusml}
\partial^{l+m}_l(v) \in \Ho^0(X^{\leq w}_{l+m},F_1\sh{D}_{l+m
\rightarrow l}(\lambda))\;,
\end{align}
see \cite{KT95}[Lemma 2.3.1]. Here $F_1 \sh{D}_{l+m \to l}(\lambda)$ denotes
the first term of the order filtration of the transfer bimodule associated to 
$p^{l+m}_l$. In this way $v$ defines
a $\C$-linear map $\Ho^j(X^{\leq w}_{l+m},\sh{M}_{l+m}) \rightarrow \Ho^j(X^{\leq w}_l,\sh{M}_l)$ and hence
an element of $\End_{\C}\Ho^j(X,\sh{M})$. 
For a $\mathfrak{h}$-module $M$ let
\begin{align*}
M_{\mu} = \{m \in M\; \vert\; \exists N \in \Z_{\geq 0}\; (h-\mu(h))^N m = 0\; \forall h \in \mathfrak{h}\}
\end{align*}
be the generalized weight space associated to $\mu \in \mathfrak{h}^*$. 
We denote $\overline{\Ho^j}(X,\sh{M}) = \bigoplus_{\mu \in \mathfrak{h}^*}\Ho^j(X,\sh{M})_{\mu}$. 
This is a $\mathfrak{g}$-submodule of $\Ho^j(X,\sh{M})$. 

\section{Intersections of Schubert cells $X_w \cap s_i X_w$}\label{ssec:intersection}
In this section we assume $s_iw < w$. 
\eqref{eq:NXw} gives $T$-equivariant isomorphisms
\begin{align}\label{eq:parametrizationofNXw}
U^-_w\times U(\alpha_i)\times ^{s_i}U_{s_iw}
\xrightarrow{\cong} N(X_w)
\end{align}
and
\begin{align}\label{eq:parametrizationofNXsw}
\begin{split}
^{s_i}U^-_w \times U^-(-\alpha_i)\times U_{s_i w} \xrightarrow{\cong} N(X_{s_iw})\;.
\end{split}
\end{align}
Here $^{s_i}(\cdot) = s_i (\cdot) s_i^{-1}$ denotes the conjugate subgroup. 
The following lemma describes the intersection of the Schubert cell $X_w$ with
its translate $s_i X_w$. 

\begin{Lemma} \label{Lemma:intersectionofXwandXsiW}
We have $X_w \cap s_i X_w  = s_i X_w \setminus X_{s_i w}= X_w \setminus s_i X_{s_i w}$. (Here $\setminus$
denotes the difference of sets.)
Under the isomorphisms \eqref{eq:parametrizationofNXw} and \eqref{eq:parametrizationofNXsw}
the identity map $s_i X_w \setminus X_{s_i w} \rightarrow X_w \setminus s_i X_{s_i w}$ is the isomorphism
\begin{align}\label{eq:transfunction}
\begin{split}
& (U^-(-\alpha_i)\setminus 1)\times U_{s_i w} \rightarrow (U(\alpha_i)\setminus 1)\times ^{s_i}U_{s_i w}\\
& (e^{zf_i}, h_i(z)^{-1}u h_i(z)) \mapsto (e^{z^{-1}e_i}, \dot{s_i}^{-1} \widetilde{u}(z) \dot{s_i})\;.
\end{split}
\end{align}
Here $z \in \Gm$ and $\dot{s_i}=e^{e_i}e^{-f_i}e^{e_i}$. $h_i$
is considered as a group homomorphism $\Gm \rightarrow T$. Given $u$ and $z$, $\widetilde{u}(z) \in U_{s_i w}$ is uniquely determined by the condition
$e^{z e_i} u \in \widetilde{u}(z) U(\Phi^{>0}\cap s_i w \Phi^{>0})$. 
\end{Lemma}

\begin{proof}
\emph{1. $X_w \cap s_i X_w \subseteq X_w \setminus s_i X_{s_i w}$
and $X_w \cap s_i X_w \subseteq  s_i X_w \setminus X_{s_i w}$.}
Let $e^{x e_i} v_2 w B = e^{z f_i} u_2 s_i w B \in X_w \cap s_i X_w$,
where $x, z \in \C$, $u_2 \in U_{s_i w}$,
$v_2 \in ^{s_i}U_{s_i w}$. If $z = 0$ the rhs is in $X_{s_i w}$
and because of $X_{s_i w} \cap X_w = \varnothing$ this is impossible. 
If $x=0$ the lhs is in $s_i X_{s_i w}$ and again this is impossible.\newline
\emph{2. $X_w \cap s_i X_w = s_i X_w \setminus X_{s_i w}= X_w \setminus s_i X_{s_i w}
$ and the given map corresponds to the identity of $X_w \cap s_i X_w$.}\newline
We check $e^{z f_i}h_i(z)^{-1}uh_i(z) s_i wB = e^{z^{-1}e_i}\dot{s_i}^{-1}\widetilde{u}(z)s_iwB$. This shows that the given map corresponds to the identity map of $X_w \cap s_i X_w$ and in particular it is an isomorphism. Indeed, we have
$e^{z^{-1}e_i}\dot{s_i}^{-1}e^{ze_i} h_i(z) = e^{z f_i}$.  
It follows
\begin{align*}
e^{z f_i} h_i(z)^{-1} u h_i(z) s_i w B = e^{z^{-1}e_i}\dot{s_i}^{-1} e^{ze_i} u s_i w B
\in e^{z^{-1}e_i}\dot{s_i}^{-1}\widetilde{u}(z)U(\Phi^{>0}\cap s_iw \Phi^{>0})s_i wB\\
= e^{z^{-1}e_i}\dot{s_i}^{-1}\widetilde{u}(z)s_iw (s_iw)^{-1}U(\Phi^{>0}\cap s_iw \Phi^{>0})s_i wB
= e^{z^{-1}e_i}\dot{s_i}^{-1}\widetilde{u}(z)s_iwB
\end{align*}
as $(s_i w)^{-1}U(\Phi^{>0}\cap s_iw\Phi^{>0})s_i w= U((s_i w)^{-1}\Phi^{>0} \cap \Phi^{>0})
\subseteq B$.
\end{proof}

\section{$\sh{D}$-module direct images from $X_w \cap s_i X_w$}\label{sec:constructionsforXwandsXw}
In this section we assume $s_i w < w$. We consider the embedding 
\begin{align}\label{eq:iwl}
i_{w,l}: X_w \cap s_i X_w \hookrightarrow X^{\leq w}_l
\end{align} 
(for $l$ large enough). We note that $i_{w,l}$ is locally closed and affine and identify its domain of definition with $\Gm \times ^{s_i}U_{s_i w}$ in view of Lemma \autoref{Lemma:intersectionofXwandXsiW}. We define analogues of
$\sh{B}_w(\lambda)_l$ and $\sh{M}_w(\lambda)_l$ of \cite{KT95}[(3.1.1) and (3.1.3)]. 

\begin{Def}\label{Def:DmodR*!wlambdaalphal}
For $\lambda \in P$, $\mu \in \C$, and $? \in \{*,!\}$ we define the right $\sh{D}_{X^{\leq w}_l}(\lambda)$-module
as
\begin{align*}
\sh{R}_{?w}(\lambda,\mu)_l = i_{w,l?}\left(\left(\Omega^{(\mu)}_{\Gm} \boxtimes \Omega_{^{s_i}U_{s_i w}}\right) \otimes i_{w,l}^*\sh{O}_{X^{\leq w}_l}(\lambda) \right)\;. 
\end{align*}
\end{Def}
Here we introduced the right $\sh{D}_{\Gm}$-module $\Omega^{(\mu)}_{\Gm}
=\sh{D}_{\Gm}/(x\partial_x-\mu)\sh{D}_{\Gm}$ as in \cite{Eic15}.
The coordinate $x$ on $\Gm$ is the one of $U(\alpha_i)$. Thus $x=\infty$ corresponds to $X_{s_i w}$. 
By definition $\sh{R}_{?w}(\lambda,\mu)_l$ is the direct image of a line bundle
with flat connection on $X_w \cap s_i X_w$ depending on $\mu$ (isomorphic to $\Omega_{X_w \cap s_i X_w}$ if and only if $\mu \in \Z$). 
\begin{Rem}
As a direct consequence of \cite{Eic15}[Remark 2.3]  we have 
\begin{align*}
\sh{R}_{?w}(\lambda,\mu+1)_l \cong \sh{R}_{?w}(\lambda,\mu)_l\qquad 
\Dual \sh{R}_{*w}(\lambda,\mu)_l \cong \sh{R}_{!w}(\lambda,-\mu)_l\;. 
\end{align*}
\end{Rem}
By Proposition \autoref{Prop:equivalencebetweenholonomiccats}(2) we see that $\sh{R}_{?w}(\lambda,\mu)=(\sh{R}_{?w}(\lambda,\mu)_l)_{l \geq l_0}$
is an object of $\Hol(\lambda,\overline{X_w},X^{\leq w},l_0)$. 

\begin{Not}
In case $\mu \in \Z$ we abbreviate 
$\sh{R}_{?w}(\lambda)_l = \sh{R}_{?w}(\lambda,\mu)_l$
and $\sh{R}_{?w}(\lambda)=\sh{R}_{?w}(\lambda,\mu)$. 
Thus, these are isomorphism classes rather than objects. 
\end{Not}

\section{$\overline{\Ho^0}(X, \sh{R}_{*w}(\lambda,\mu))$ as $\mathfrak{h}$-module}
\label{sec:H0ofR*wlambdaalphaashmodule}
The following theorem is similar to \cite{KT95}[Lemma 3.2.1]. 

\begin{Thm}\label{Thm:H0hmod} We have isomorphisms of $\mathfrak{h}$-modules
\begin{enumerate}
\item $\Ho^0(X^{\leq w}_l,\sh{R}_{*w}(\lambda,\mu)_l) 
\cong \C[z,z^{-1}]\otimes_{\C} \Sym (\mathfrak{n}^-_i /\mathfrak{n}^-_l)\otimes_{\C} \C_{s_i w\cdot \lambda + \mu\alpha_i}$
\item $\overline{\Ho^0}(X,\sh{R}_{*w}(\lambda,\mu)) \cong 
\C[z,z^{-1}]\otimes_{\C} \Sym \mathfrak{n}^-_i \otimes_{\C} \C_{s_i w\cdot \lambda + \mu\alpha_i}$
\end{enumerate}
Here $z$ has weight $\alpha_i$. 
\end{Thm}
\begin{proof}
(1) The $\mathfrak{h}$-module $\Ho^0(X^{\leq w}_l,\sh{R}_{*w}(\lambda,\mu)_l)$ can be
computed in the same way as the 
$\mathfrak{h}$-module $\overline{\Ho^0}(N(X_{s_iw})_l,\sh{M}_w(\lambda)_l)$ is computed
in \cite{KT95}[(3.6.4)], as we will explain. Here $N(X_{s_iw})_l = U^-_l \backslash N(X_{s_iw})$
is open in $X^{\leq w}_l$. Since $i_{w,l}$ is the composition of
\begin{align*}
\kappa_{w,l}: X_w \cap N(X_{s_iw})_l = X_w \cap s_iX_w \hookrightarrow N(X_{s_iw})_l
\end{align*} 
with the open embedding $N(X_{s_iw})_l \hookrightarrow X^{\leq w}_l$, we have 
\begin{align}\label{eq:R*wrestrtoNXswisBw}
\sh{R}_{*w}(\lambda)_l \vert N(X_{s_iw})_l 
\cong \kappa_{w,l*}(\Omega_{X_w \cap s_iX_w}\otimes i_{w,l}^*\sh{O}_{X^{\leq w}_l}(\lambda))
\cong \sh{B}_w(\lambda)_l\vert N(X_{s_iw})_l\;.
\end{align}
Replacing $\Omega_{\Gm}$ by
$\Omega^{(\mu)}_{\Gm}$ in the computation leading to the expression \cite{KT95}[(3.6.4)] gives
\begin{align*}
& \Ho^0(N(X_{s_iw})_l,\sh{R}_{*w}(\lambda,\mu)_l) \\ 
& \cong \Uea \Lie(^{s_i}U^-_{w,l})\otimes_{\C} \Ho^0\left(\Gm,\Omega^{(\mu)}_{\Gm}\right)
\otimes_{\C}\Ho^0(U_{s_iw},\sh{O}_{U_{s_iw}})\otimes_{\C}\C\omega\otimes_{\C}\C t^{s_iw,\lambda}_l
\end{align*}
as $\mathfrak{h}$-module. We have denoted $^{s_i} U^-_{w,l} = U_l^- \backslash ^{s_i} U_w^-$. 
Here $\omega$ is 
a nowhere vanishing section of $\Ho^0(U_{s_iw},\Omega_{U_{s_iw}})$ and $t^{s_iw,\lambda}_l
\in \Ho^0(N(X_{s_iw})_l,\sh{O}_{X^{\leq w}_l}(\lambda))$ is the nowhere vanishing
section defined in \cite{KT95}[(3.2.4)]. 
Then $1 \otimes \overline{1}z^n \otimes 1 \otimes \omega \otimes t^{s_iw,\lambda}_l$
has weight $s_i w \cdot \lambda + (\mu+n)\alpha_i$. 
We simplify this to
\begin{align*}%\label{eq:sectionsofRwlambdaalphaoverNXsiwsimplified}
\Ho^0(N(X_{s_iw})_l,\sh{R}_{*w}(\lambda,\mu)_l) 
\cong \C[z,z^{-1}]\otimes_{\C} \Sym (\mathfrak{n}^-_i/\mathfrak{n}^-_l)\otimes_{\C} \C_{s_i w\cdot \lambda + \mu \alpha_i}\;, 
\end{align*}
where $z$ is given the weight $\alpha_i$. 
The restriction maps give isomorphisms
\begin{align*}%\label{eq:resisomforR*}
 \Ho^0(N(X_w)_l,\sh{R}_{*w}(\lambda,\mu)_l) \xleftarrow{\cong} \Ho^0(X^{\leq w}_l, \sh{R}_{*w}(\lambda,\mu)_l)\xrightarrow{\cong} \Ho^0(N(X_{s_iw})_l,\sh{R}_{*w}(\lambda,\mu)_l)\;. 
\end{align*}
We remark that the corresponding statement for $\sh{R}_{!w}(\lambda,\mu)_l$ is not a priori true. \\
(2) directly follows from (1). 
\end{proof}

\section{Cohomology vanishing}\label{sec:cohomologyvanishing}
The following lemma and its proof is similar to \cite{KT95}[Lemma 3.2.1(ii)] and its proof.
\begin{Lemma}\label{Lemma:cohvanishingforR*wlambdaalpha}
$\Ho^j(X^{\leq w}_l,\sh{R}_{*w}(\lambda,\mu)_l)=0$ and
consequently $\overline{\Ho^j}(X,\sh{R}_{*w}(\lambda,\mu))=0$ for $j>0$.\qed
\end{Lemma}

\section{$\mathfrak{g}$-modules $\RVerma(\lambda,\mu)^{(\veesl2)}$}\label{sec:gmodulesRVermalambdalpha(veesl2)}
\subsection{Definition of $\RVerma(\lambda,\mu)^{(\veesl2)}$}
\begin{Def}\label{Def:RVermalambdalapha}
In \cite{Eic15} we introduced right weight $\mathfrak{sl}_2$-modules depending on two parameters. We use them in the present definition, denoting the module $\RVerma(\Lambda,\mu)^{(\vee)}$ of loc. cit. by $\RVerma^{\mathfrak{sl}_2}(\Lambda,\mu)^{(\veesl2)}$ and considering it as a left module via the
anti-involution $-\id$ of $\mathfrak{sl}_2$. We define the $\mathfrak{g}$-module $\RVerma(\lambda,\mu)$ and $\RVerma(\lambda,\mu)^{\veesl2}$
\begin{align*}
\RVerma(\lambda,\mu)^{(\veesl2)} =
\Uea\mathfrak{g}\otimes_{\Uea \mathfrak{p}_i}\left(\C_{\lambda}\otimes_{\C}\RVerma^{\mathfrak{sl}_2}(\lambda(h_i),\mu)^{(\veesl2)}\right)
\end{align*}
for $\lambda \in P$ and $\mu \in \C$.
Here $\mathfrak{p}_i$ acts on $\C_{\lambda}\otimes_{\C}\RVerma^{\mathfrak{sl}_2}(\lambda(h_i),\mu)^{(\veesl2)}$ via the projection 
\begin{align}\label{eq:projfrompitogi}
\mathfrak{p}_i \twoheadrightarrow \mathfrak{p}_i/\mathfrak{n}_i = \mathfrak{g}_i = 
\{h \in \mathfrak{h}\; \vert\; \alpha_i(h)=0\}\oplus \mathfrak{g}_i^{\prime}\;.
\end{align}
Moreover $\mathfrak{g}_i^{\prime}$ is identified with $\mathfrak{sl}_2$ by $e_i \mapsto e,
f_i \mapsto f, h_i \mapsto h$ (the modules $\RVerma^{\mathfrak{sl}_2}(\lambda(h_i),\mu)^{(\veesl2)}$
depend on the choice of $h$). 
\end{Def}
Here we use the notation $(\cdot)^{\veesl2}$ in order to distinguish it from
the duality functor $(\cdot)^{\vee}$ of \cite{KT95}[Section 1.2], denoted $(\cdot)^*$ there,
which we will also need.\par

It follows that the vector $1\otimes(1\otimes 1) \in \RVerma(\lambda,\mu)$ has weight 
$\lambda+\mu \alpha_i$. Indeed, for $g \in \mathfrak{h}$ we have, applying \eqref{eq:projfrompitogi},
\begin{align*}
g(1\otimes (1\otimes 1))&=1\otimes((g-\frac{1}{2}\alpha_i(g)h_i)1)\otimes 1+1\otimes(1\otimes \frac{1}{2}\alpha_i(g)h_i 1)\;. 
\end{align*}
By the defining relations for $\RVerma^{\mathfrak{sl}_2}(\lambda(h_i),\mu)$ given
in \cite{Eic15}
and our conventions we have
\begin{align*}
1 \otimes (1\otimes h_i 1)= -1\otimes (1\otimes 1h)=(2\mu+\lambda(h_i))1\otimes(1\otimes 1) 
\end{align*}
and it follows that $g(1\otimes (1\otimes 1))=(\lambda+\mu\alpha_i)(g)(1\otimes(1\otimes 1))$. The Lie subalgebra $\mathfrak{n}_i$
annihilates the $\mathfrak{g}_i$-submodule $\C 1\otimes \C_{\lambda}\otimes \RVerma^{\mathfrak{sl}_2}(\lambda(h_i),\mu)^{(\veesl2)}$ of $\RVerma(\lambda,\mu)^{(\veesl2)}$ due to \eqref{eq:projfrompitogi}.
By the above computation of the weight of $1\otimes(1\otimes 1)$ we have an isomorphism of $\mathfrak{h}$-modules
\begin{align}\label{eq:toplevelRlambdaalphahmod}
\C_{\lambda}\otimes \RVerma^{\mathfrak{sl}_2}(\lambda(h_i),\mu)^{(\veesl2)} \cong \C[z,z^{-1}]\otimes_{\C}\C_{\lambda+\mu \alpha_i}\;,
\end{align}
where $z$ has weight $\alpha_i$, sending $1\otimes 1^{(\vee)} \mapsto 1\otimes 1$. 
For the case of $(\cdot)^{\veesl2}$ we used the fact that the
duality $(\cdot)^{\veesl2}$ of $\mathfrak{sl}_2$-modules defined in \cite{Eic15} does not
change the $\C h_i$-module. 
By the PBW theorem $\Uea\mathfrak{g} \cong \Uea\mathfrak{n}^-_i \otimes_{\C} \Uea\mathfrak{p}_i$ as
left $\mathfrak{n}^-_i$-module and hence the action of $\mathfrak{n}_i^-$ on this submodule  is free and generates the module. The
isomorphism of $\mathfrak{h}$-modules 
\begin{align}\label{eq:RVermalambdaalpha(vee)hmod}
\RVerma(\lambda,\mu)^{(\veesl2)} \cong \Sym \mathfrak{n}^-_i \otimes_{\C} \C[z,z^{-1}]\otimes_{\C}\C_{\lambda+\mu \alpha_i}
\end{align}
sending $1\otimes (1\otimes 1^{(\vee)}) \mapsto 1\otimes 1\otimes 1$ directly follows from \eqref{eq:toplevelRlambdaalphahmod}. From Theorem \autoref{Thm:H0hmod}(2) and \eqref{eq:RVermalambdaalpha(vee)hmod} we see that  we have an isomorphism of $\mathfrak{h}$-modules
$\overline{\Ho^0}(X,\sh{R}_{*w}(\lambda,\mu)) \cong \RVerma(s_i w \cdot \lambda, \mu)^{(\veesl2)}
\cong \RVerma(w \cdot \lambda, \mu)^{(\veesl2)}$ for any $\lambda \in P$.

\begin{Not}
We introduce $\RVerma(\lambda,?)^{(\veesl2)}$ as the isomorphism classes of $\RVerma(\lambda,\mu)^{(\veesl2)}$
 for $\mu \in \Z$ by considering
in Definition \autoref{Def:RVermalambdalapha} the corresponding isomorphism class $\RVerma^{\mathfrak{sl}_2}(\lambda(h_i),?)^{(\veesl2)}$ of \cite{Eic15}, where $? \in \{=,<,>\}$. 
This is convenient since it eliminates the parameter $\mu$ from the notation when $\mu \in \Z$. 
\end{Not}

Let us also recall the standard notion of the \emph{Verma module} $\MVerma(\lambda) 
= \Uea\mathfrak{g} \otimes_{\Uea\mathfrak{b}}\C_{\lambda}$ of highest weight $\lambda \in \mathfrak{h}^*$,
where $\mathfrak{b}$ acts via $\mathfrak{b} \twoheadrightarrow \mathfrak{b}/\mathfrak{n}=\mathfrak{h}$
on $\C_{\lambda}$. 

\subsection{Comments on the literature and terminology}\label{ssec:commentsonlit}
The $\mathfrak{g}$-modules $\Uea\mathfrak{g}\otimes_{\Uea\mathfrak{p}_i}M$, where
$M$ is a finite dimensional simple $\mathfrak{g}_i$-module considered as a $\mathfrak{p}_i$-module
via \eqref{eq:projfrompitogi}, are called parabolic Verma modules, see e.g. \cite{Hum08}[Section 9.4].\par 
At least in the case $\mathfrak{g}=\widehat{\mathfrak{sl}_2}$ ($I=\{0,1\}$) the modules
$\RVerma(\lambda,\mu)$ have appeared in the literature.  
The $\widehat{\mathfrak{sl}_2}$-modules $\RVerma(\lambda,\mu)$ were
introduced and studied under the name \emph{relaxed Verma modules} in \cite{SS97}, \cite{FST98}, \cite{FSST98}.
Indeed, in terms of the loop algebra realization $\widehat{\mathfrak{sl}_2}
=(\mathfrak{sl}_2 \otimes_{\C} \C[t,t^{-1}]) \oplus \C K \oplus \C d$ and for $i=1$, we can describe
$\RVerma(\lambda,\mu)$ as being generated by a vector $v$ satisfying
\begin{align*}
(\mathfrak{sl}_2 \otimes t \C[t])v=0\quad (f\otimes 1)(e\otimes 1)v=\mu v \quad (h\otimes 1)v=j v\quad K v=kv\quad dv=\delta v\;. 
\end{align*}
Here the parameters $\mu, j, k, \delta \in \C$ are determined by $\lambda$ and $\mu$. This is called relaxed highest
weight condition. Furthermore, the action of the Lie subalgebra $\mathfrak{sl}_2 \otimes t^{-1}\C[t^{-1}]$ of $\widehat{\mathfrak{sl}_2}$ 
on the $(\mathfrak{sl}_2 \otimes \C 1) \oplus \C K \oplus \C d$-submodule  
\begin{align*}
\bigoplus_{n \geq 0} \C(f\otimes 1)^n v \oplus \bigoplus_{n \geq 1} \C(e\otimes 1)^n v
\end{align*} generates $\RVerma(\lambda,\mu)$ freely. 
In \cite{Rid10}, \cite{Fje11}, \cite{RW15}  and elsewhere these modules have recently been investigated in the context of two-dimensional
conformal field theory. The name relaxed is the reason for our choice of the letter R in the notation. 
We propose to call all modules $\RVerma(\lambda,\mu)^{(\veesl2)}$ of Definition \autoref{Def:RVermalambdalapha} and more generally $\mathfrak{g}$-modules on which the Lie subalgebra $\mathfrak{n}_i$ acts locally nilpotently \emph{relaxed highest weight modules}. This explains the title of this article. 

\section{Auto-equivalence $\widetilde{s_i}_*$ of $\Hol(\lambda)$}
\label{sec:exactautoequivwidetildesi*ofHol(lambda)} 

The element $\widetilde{s_i} \in \widetilde{W}$, $i \in I$, acts by
the automorphism $e^{e_i}e^{-f_i}e^{e_i}$ of the Lie algebra $\mathfrak{g}$.
We will denote this automorphism by $\widetilde{s_i}$, too. 
\begin{Def}\label{Def:twistbywidetildesi}
In the standard manner $\widetilde{s_i}$ induces  a functor $(\cdot)^{\widetilde{s_i}}: \mathfrak{g}-\modulecat \rightarrow \mathfrak{g}-\modulecat$. Namely, 
for $M \in \mathfrak{g}-\modulecat$ we define $M^{\widetilde{s_i}} \in \mathfrak{g}-\modulecat$ by
$v m = \widetilde{s_i}(v)m$ for $m \in M$, $v \in \mathfrak{g}$, where on the rhs we mean
the $\mathfrak{g}$-action on $M$. 
\end{Def}
Thus, by definition $M$ and $M^{\widetilde{s_i}}$ have the same underlying vector space. 
$(\cdot)^{\widetilde{s_i}}$ is an auto-equivalence of $\mathfrak{g}-\modulecat$. 

$\widetilde{s_i}$ defines automorphisms of schemes $\widetilde{s_i}: G \rightarrow G$ 
and $\widetilde{s_i}: X \rightarrow X$. \emph{We will abbreviate $\widetilde{s_i}$
to $s$ in this section and denote by the same letter the simple reflection $s_i$, its image in the quotient $W$. (It will be clear from the context which object we will be referring to.)} 
\begin{Lemma}\label{Lemma:actionofsonnminusl}
We have $\mathfrak{n}^-_{l+d} \subseteq s\mathfrak{n}^-_l \subseteq \mathfrak{n}^-_{l-d}$
for $d \geq \Delta=4$ and $l-d \geq 0$. 
\end{Lemma}
\begin{proof}
This is a direct application of \cite{Kac90}[Exercise 5.13, 5.14]. 
\end{proof}
The number $\Delta$ will occur as a parameter in the constructions below. 
We note that for $w$ such that $sw < w$ we have $P_i ^- X^{\leq w}=X^{\leq w}$,
thus $s X^{\leq w}=X^{\leq w}$, and $s \overline{X_w} = \overline{X_w}$. 
\emph{In this section we will abbreviate $Y=X^{\leq w}$ for $w$ such that $sw < w$.}

\subsection{Morphisms $s^l_{l^{\prime}}$}
Lemma \autoref{Lemma:actionofsonnminusl} implies that 
$^s U^-_l$, the $s$-conjugate of $U^-_l$, is a closed subgroup scheme of $U^-_{l^{\prime}}$ for $l \geq l^{\prime}+\Delta$ . Hence we have a quotient morphism $^s U^-_l \backslash Y \rightarrow Y_{l^{\prime}}$
($l^{\prime}$ large enough), which is a morphism whose fibers are isomorphic to a finite dimensional affine space, in particular it is an affine morphism. 

\begin{Def}\label{Def:sllprime}
We define $s^l_{l^{\prime}}$ as the composition
\begin{align*}
s^l_{l^{\prime}}: Y_l \xrightarrow{\cong} ^s U^-_l \backslash Y
\rightarrow Y_{l^{\prime}}\;,
\end{align*}
where the first morphism is the isomorphism induced by $s$.
\end{Def}
Thus $p_{l^{\prime}}s=s^l_{l^{\prime}} p_l$ holds. Clearly
$s^l_{l^{\prime}}p^k_l = s^k_{l^{\prime}}$ and $p^{k^{\prime}}_{l^{\prime}}s^k_{k^{\prime}}=s^k_{l^{\prime}}$
holds for $k \geq l$ and $k^{\prime} \geq l^{\prime}$. 

\begin{Rem}\label{Rem:pullbackofOlambdabys}
There is an isomorphism $s^* \sh{O}_X(\lambda) \xrightarrow{\cong} \sh{O}_X(\lambda)$ of line bundles
sending a local section $\varphi \mapsto \varphi(s\cdot)$. We will 
consider the isomorphism of sheaves $1_s: s^{-1}\sh{O}_X(\lambda)\vert Y \xrightarrow{\cong} \sh{O}_X(\lambda) \vert Y$.
Similarly, there are isomorphisms $(s^l_{l^{\prime}})^* \sh{O}_{Y_{l^{\prime}}}(\lambda) \xrightarrow{\cong} \sh{O}_{Y_l}(\lambda)$ for $l \geq l^{\prime}+\Delta$ sending
the distinguished local section $t^{sv,\lambda}_{l^{\prime}}$
of \cite{KT95}[(3.2.4)] to a nonzero multiple of $t^{v,\lambda}_l$, where $v \leq w$. These isomorphisms are compatible with the 
isomorphisms $p^{k*}_l\sh{O}_{Y_l}(\lambda) \xrightarrow{\cong} \sh{O}_{Y_k}(\lambda)$ 
of \eqref{eq:pullback}. 
We will also consider the monomorphism
of sheaves $1_{s^l_{l^{\prime}}}: (s^l_{l^{\prime}})^{-1}\sh{O}_{Y_{l^{\prime}}}(\lambda)
\hookrightarrow \sh{O}_{Y_l}(\lambda)$
corresponding to $(s^l_{l^{\prime}})^* \sh{O}_{Y_{l^{\prime}}}(\lambda) \xrightarrow{\cong} \sh{O}_{Y_l}(\lambda)$. 

\end{Rem}

\subsection{Definition of $s_*$}

\begin{Def}\label{Def:s*Ml}
Let $\sh{M} \in \Hol(\lambda)$. Then there is a $w \in W$ with $sw < w$ such
that $\sh{M} \in \Hol(\lambda,\overline{X_w},X^{\leq w},l_0)$. 
We define $(s_* \sh{M})_l = s^{l+\Delta}_{l*} \sh{M}_{l+\Delta}$ for $l \geq l_0$.  
\end{Def}

\begin{Lemma}\label{Lemma:s*exactautoequivalenceofHol(lambda)}
$s_*$ defines an auto-equivalence of $\Hol(\lambda)$. 
\end{Lemma}
\begin{proof}
The functor $s^{l+\Delta}_{l*}$ is an
 equivalence 
\begin{align}\label{eq:functorslplusDeltal}
\Hol\left(\sh{D}^{(s^{l+\Delta}_l)^*\sh{O}_{Y_l}(\lambda)}_{Y_{l+\Delta}},\overline{X_w}\right)
\rightarrow \Hol\left(\sh{D}^{\sh{O}_{Y_l}(\lambda)}_{Y_l},\overline{X_w}\right) 
\end{align} 
by Proposition \eqref{Prop:equivalencebetweenholonomiccats}(1). 
We identify the source 
category with $\Hol\left(\sh{D}^{\sh{O}_{Y_{l+\Delta}}(\lambda)}_{Y_{l+\Delta}},\overline{X_w}\right)$
by the isomorphism of line bundles described in Remark \autoref{Rem:pullbackofOlambdabys}. The isomorphisms 
\eqref{eq:gammakl} induce isomorphisms 
\begin{align}\label{eq:isomorphismsdeltakl}
\delta^k_l=s^{l+\Delta}_{l*}(\gamma^{k+\Delta}_{l+\Delta}): p^k_{l*}(s_*\sh{M})_k \xrightarrow{\cong}(s_*\sh{M})_l
\end{align}
 and one checks that the latter again satisfy the coherence condition \eqref{eq:gammaformtransitivesystem}. It follows
that $((s_*\sh{M})_l)_{l \geq l_0} \in \Hol(\lambda,\overline{X_w},X^{\leq w},l_0)$. 
It is then clear, since the construction does
not depend on the choice of $w$, that $s_*$ indeed defines a functor $\Hol(\lambda) \rightarrow \Hol(\lambda)$. 
Due to the equivalences \eqref{eq:equivalenceHol} the fact that \eqref{eq:functorslplusDeltal} is an  equivalence implies that $s_*$ is an equivalence. 
\end{proof}

\begin{Rem}
In Definition \autoref{Def:s*Ml} we can replace $l+\Delta$ by a more general function of $l$.
The resulting functor will be isomorphic to $s_*$. 
\end{Rem}

Proposition \autoref{Prop:equivalencebetweenholonomiccats}(3) implies

\begin{Lemma}\label{Lemma:Holonomicdualitycommuteswithwidetildesi*} $\Dual s_* = s_* \Dual$
holds on $\Hol(\lambda)$.\qed
\end{Lemma}

\subsection{Cohomology groups $\Ho^j(X,s_*\sh{M})$}
We recall the transfer $(\sh{D}_{Y_{l+m}}(\lambda),(p^{l+m}_l)^{-1}\sh{D}_{Y_l}(\lambda))$-bimodule
\begin{align*}
\sh{D}_{Y_{l+m} \to Y_l}(\lambda)
\subseteq \sh{H}om_{\C}((p^{l+m}_l)^{-1}\sh{O}_{Y_l}(-\lambda),
\sh{O}_{Y_{l+m}}(-\lambda))
\end{align*}
for the morphism $p^{l+m}_l$ (here $m \geq 0$), see e.g. \cite{KT95}[(2.1.9)].  
\emph{In this and the next section we will abbreviate $\sh{O}_{Y_l}(-\lambda)$ by 
$\sh{O}_l$, $\sh{D}_{Y_l}(\lambda)$ by $\sh{D}_l$ and $\sh{D}_{Y_{l+m} \to Y_{l}}(\lambda)$ by 
$\sh{D}_{l+m \rightarrow l}$ in order to make the expressions more readable.} 
In the same way as we introduced $s^l_{l^{\prime}}$ in Definition \autoref{Def:sllprime} we construct
the morphisms $t^l_{l^{\prime}}: Y_l \rightarrow Y_{l^{\prime}}$, $l \geq l^{\prime}+\Delta$,
from the automorphism $t=s^{-1}:X \rightarrow X$ inverse to $s$. This again
works due to Lemma \autoref{Lemma:actionofsonnminusl}. We then have
$t^l_{l^{\prime}}s^k_l = p^k_{l^{\prime}}=s^l_{l^{\prime}}t^k_l$. 
We define the morphism of sheaves 
\begin{align}\label{eq:defofepsilon}
\epsilon: (s^{k^{\prime}}_{l^{\prime}})^{-1}\sh{D}_{l^{\prime}
\rightarrow l} \rightarrow \sh{D}_{k^{\prime} \rightarrow k},\; p
\mapsto \epsilon(p)\;,
\end{align}
where $\epsilon(p)$ is the composition
\begin{align*}%\label{eq:defofepsilon(p)}
(p^{k^{\prime}}_k)^{-1}\sh{O}_k
\xrightarrow{(s^{k^{\prime}}_l)^{-1}1_{t^l_k}}
(s^{k^{\prime}}_l)^{-1}\sh{O}_l \xrightarrow{p}
(s^{k^{\prime}}_{l^{\prime}})^{-1}\sh{O}_{l^{\prime}} \xrightarrow{1_{s^{k^{\prime}}_{l^{\prime}}}}
\sh{O}_{k^{\prime}}\;. 
\end{align*}
This is defined whenever $k^{\prime} \geq l^{\prime}+\Delta$, $l \geq k+\Delta$,
$l^{\prime} \geq l$. We used the morphism of sheaves  $1_{s^{k^{\prime}}_{l^{\prime}}}$
defined in Remark \autoref{Rem:pullbackofOlambdabys} and the similar morphism 
$1_{t^l_k}$ associated to $t^l_k$. 
$\epsilon$ is a monomorphism. It can be understood as the 
embedding of the middle factor into the tensor product under the isomorphism of bimodules
\begin{align*}
\sh{D}_{s^{k^{\prime}}_{l^{\prime}}}\otimes_{(s^{k^{\prime}}_{l^{\prime}})^{-1}\sh{D}_{l^{\prime}}}
(s^{k^{\prime}}_{l^{\prime}})^{-1}\sh{D}_{l^{\prime}\to l}\otimes_{(s^{k^{\prime}}_l)^{-1}\sh{D}_l}
\otimes (s^{k^{\prime}}_l)^{-1}\sh{D}_{t^l_k} \xrightarrow{\cong} \sh{D}_{k^{\prime} \to k}
\end{align*}
induced by the factorization $p^{k^{\prime}}_k = t^l_k p^{l^{\prime}}_l s^{k^{\prime}}_{l^{\prime}}$ (and also defined via composition), cf. \cite{KT95}[(2.3.2)]. Here $\sh{D}_{s^{k^{\prime}}_{l^{\prime}}}$ and
$\sh{D}_{t^l_k}$ is the transfer bimodule associated to $s^{k^{\prime}}_{l^{\prime}}$
and $t^l_k$ respectively and we have the relation
\begin{align*}
1_{s^{k^{\prime}}_{l^{\prime}}} \in \sh{D}_{s^{k^{\prime}}_{l^{\prime}}}
\subseteq \sh{H}om_{\C}((s^{k^{\prime}}_{l^{\prime}})^{-1}\sh{O}_{l^{\prime}},
(s^{k^{\prime}}_{l^{\prime}})^*\sh{O}_{l^{\prime}}) \cong \sh{H}om_{\C}((s^{k^{\prime}}_{l^{\prime}})^{-1}\sh{O}_{l^{\prime}},
\sh{O}_{k^{\prime}})
\end{align*}
and similarly
\begin{align*}
1_{t^l_k} \in \sh{D}_{t^l_k} \subseteq \sh{H}om_{\C}((t^l_k)^{-1}\sh{O}_k,
(t^l_k)^*\sh{O}_k)\cong \sh{H}om_{\C}((t^l_k)^{-1}\sh{O}_k,\sh{O}_l)\;. 
\end{align*}
Let $\sh{M} \in \Hol(\lambda,\overline{X_w},X^{\leq w},l_0)$. 
The isomorphisms
\begin{align*}
\gamma^{l+m}_l: p^{l+m}_{l*}\sh{M}_{l+m}
=p^{l+m}_{l\cdot}(\sh{M}_{l+m}\otimes_{\sh{D}_{l+m}}\sh{D}_{l+m \rightarrow l})\xrightarrow{\cong} \sh{M}_l
\end{align*}
induce $\C$-linear maps 
\begin{align*}
\widehat{\gamma}^{l+m}_l: \Ho^j(Y_{l+m},\sh{M}_{l+m})\otimes_{\Ho^0(Y_{l+m},\sh{D}_{l+m})}
\Ho^0(Y_{l+m},\sh{D}_{l+m \rightarrow l}) \rightarrow \Ho^j(Y_l,\sh{M}_l)
\end{align*}
for each $j \geq 0$. By construction of $s_*\sh{M}$ we have
similar maps $\delta^{l+m}_l$ \eqref{eq:isomorphismsdeltakl}
and the associated maps $\widehat{\delta}^{l+m}_l$. 
We define the map $(s^{l+\Delta}_l)^{\sharp}: \Ho^j(Y_{l+\Delta},\sh{M}_{l+\Delta})
\rightarrow \Ho^j(Y_l,(s_*\sh{M})_l)$ by taking $\Ho^j(Y_l,s^{l+\Delta}_{l \cdot}(\cdot))$ of 
\begin{align*}
\sh{M}_{l+\Delta} \rightarrow \sh{M}_{l+\Delta}\otimes_{\sh{D}_{l+\Delta}}\sh{D}_{s^{l+\Delta}_l},\; v \mapsto v \otimes 1_{s^{l+\Delta}_l}\;.
\end{align*}
Entirely similarly we define a map $(t^l_{l-\Delta})^{\sharp}: \Ho^j(Y_l,\sh{M}_l)
\rightarrow \Ho^j(Y_{l-\Delta},t^l_{l-\Delta *}\sh{M}_l)$.  
Since $t^l_{l-\Delta *}s^{l+\Delta}_{l*}=p^{l+\Delta}_{l-\Delta *}$ we have an isomorphism
\begin{align*}
\Ho^j(Y_{l-\Delta},t^l_{l-\Delta *}(s_*\sh{M})_l) \xrightarrow{\cong} \Ho^j(Y_{l-\Delta},\sh{M}_{l-\Delta})
\end{align*}
induced by $\gamma^{l+\Delta}_{l-\Delta}$ and denoted by the same letter. 
We note that the composition 
\begin{align*}
\gamma^{l+\Delta}_{l-\Delta}(t^l_{l-\Delta})^{\sharp} (s^{l+\Delta}_l)^{\sharp}:
\Ho^j(Y_{l+\Delta},\sh{M}_{l+\Delta}) \rightarrow \Ho^j(Y_{l-\Delta},\sh{M}_{l-\Delta})
\end{align*}
 is a transition map for the inverse limit $\Ho^j(X,\sh{M})$ (see
 \autoref{ssec:cohomologygroups}).
It follows that we have an isomorphism of vector spaces
\begin{align}\label{eq:isomofvspacesbetweenMands*M}
\Ho^j(X,\sh{M}) \xrightarrow{\cong} \Ho^j(X,s_*\sh{M})
\end{align}
induced by the maps $(s^{l+\Delta}_l)^{\sharp}$, which is
functorial in $\sh{M}$. The inverse isomorphism is induced by the
maps $\gamma^{l+\Delta}_{l-\Delta}(t^l_{l-\Delta})^{\sharp}$. The global sections of 
the morphism of sheaves \eqref{eq:defofepsilon} (for $k^{\prime}=l+m+\Delta$,
$k=l-\Delta$, $l^{\prime}=l+m$) are a map
\begin{align*} 
\epsilon: \Ho^0(Y_{l+m},\sh{D}_{l+m \rightarrow l})
\rightarrow \Ho^0(Y_{l+m+\Delta},\sh{D}_{l+m+\Delta \rightarrow l-\Delta})\;. 
\end{align*}
It satisfies the identity 
\begin{align}\label{eq:identityinvolvingepsilon}
\widehat{\gamma}^{l+m+\Delta}_{l-\Delta}(v \otimes \epsilon(p))=\gamma^{l+\Delta}_{l-\Delta}(t^l_{l-\Delta})^{\sharp}\; \widehat{\delta}^{l+m}_l
((s^{l+m+\Delta}_{l+m})^{\sharp}(v)\otimes p)
\end{align} for $v \in \Ho^j(Y_{l+m+\Delta},\sh{M}_{l+m+\Delta})$
and $p \in \Ho^0(Y_{l+m},\sh{D}_{l+m \rightarrow l})$. Indeed, this identity is implied
by the identity of isomorphisms of sheaves $\gamma^{l+m+\Delta}_{l-\Delta} = \gamma^{l+\Delta}_{l-\Delta}
t^l_{l-\Delta *}(\delta^{l+m}_l)$, which itself is a direct consequence of \eqref{eq:isomorphismsdeltakl} and \eqref{eq:gammaformtransitivesystem}. 

\subsection{$\mathfrak{g}$-action on $\Ho^j(X,s_*\sh{M})$}
The following theorem relates the $\mathfrak{g}$-action on $\Ho^j(X,s_*\sh{M})$ 
to the functor $(\cdot)^s$ of Definition \autoref{Def:twistbywidetildesi}. 

\begin{Thm}\label{Thm:Hojofs_*M}
Let $\sh{M} \in \Hol(\lambda)$. 
\eqref{eq:isomofvspacesbetweenMands*M} induces an isomorphism $\Ho^j(X,s_*\sh{M}) \cong \Ho^j(X,\sh{M})^s$ of $\mathfrak{g}$-modules. This isomorphism is functorial in $\sh{M}$. 
\end{Thm}

\begin{proof}
Let $j \in I$. 
We recall that the $P_j^-$-equivariant structure of $\sh{O}_X(-\lambda)$ induces
a Lie algebra map $\partial: \mathfrak{p}_j^- \rightarrow \End_{\C} (\sh{O}_X(-\lambda)\vert Y)$. 
Similarly, the line bundle $s^*\sh{O}_X(-\lambda)$ has a $^s P_j^-$-equivariant structure
(the group scheme $^s P_j^-$ is defined as the $s$-conjugate of $P_j^-$ in $G$; it is a subscheme of $G$). We obtain a Lie algebra map 
\begin{align*}
^s \partial: ^s \mathfrak{p}_j \rightarrow \End_{\C}(s^*\sh{O}_X(-\lambda)\vert Y)=\End_{\C}(s^{-1}\sh{O}_X(-\lambda)\vert Y)
\end{align*}
such that the diagram
\begin{align*}
\xymatrix{\mathfrak{p}_j^- \ar[r]^(0.3){\partial} & \End_{\C}(\sh{O}_X(-\lambda)\vert Y)\\
^s \mathfrak{p}_j^- \ar[u]^s_{\cong} \ar[r]^(0.3){^s\partial} & \End_{\C}(s^{-1}\sh{O}_X(-\lambda)\vert Y) \ar[u]^{\Phi}_{\cong}}
\end{align*}
commutes, where $s: ^s \mathfrak{p}_j^-
\xrightarrow{\cong} \mathfrak{p}_j^-$ is the restriction of the automorphism $s$ of
$\mathfrak{g}$ and $\Phi(p)=1_s p (1_s)^{-1}$ ($1_s$ was introduced in
Remark \autoref{Rem:pullbackofOlambdabys}). This diagram extends to a commutative
diagram
\begin{align}\label{eq:partialands-1partial}
\xymatrix{\mathfrak{g} \ar[r]^(0.3){\partial} & \End_{\C}(\sh{O}_X(-\lambda)\vert Y)\\
\mathfrak{g} \ar[u]^s_{\cong} \ar[r]^(0.3){^s\partial} & \End_{\C}(s^{-1}\sh{O}_X(-\lambda)\vert Y) \ar[u]^{\Phi}_{\cong}}\;. 
\end{align}
We note that $\Phi^{-1}=s^{-1}$, where $s^{-1}(p)$ is understood as a morphism of sheaves
for $p: \sh{O}_X(-\lambda)\vert Y \rightarrow \sh{O}_X(-\lambda)\vert Y$, holds on the image
of $\partial$. 
On the other hand, we have 
by definition \eqref{eq:defofepsilon} of $\epsilon$ a commutative diagram
\begin{align}\label{eq:mapsbetweenOandmapsbetweens-1O}
\xymatrix@C=1.7cm{p_{l-\Delta}^{-1}\sh{O}_{l-\Delta} \ar[drrr]^{1_{p_{l-\Delta}}} \ar[dr]_{s_l^{-1}(1_{t^l_{l-\Delta}})}  \ar[ddd]^{p_{l+m+\Delta}^{-1}(\epsilon(p))} \\
& s_l^{-1}\sh{O}_l \ar[r]_{s^{-1}(1_{p_l})} \ar[d]^{p_{l+m+\Delta}^{-1}(p)} & s^{-1}\sh{O}_X(-\lambda)\vert Y
\ar[r]_{1_s} \ar[r] & \sh{O}_X(-\lambda)\vert Y \\
& s_{l+m}^{-1}\sh{O}_{l+m} \ar[r]^{s^{-1}(1_{p_{l+m}})}
 \ar[dl]_{p_{l+m+\Delta}^{-1}(1_{s^{l+m+\Delta}_{l+m}})} & s^{-1}\sh{O}_X(-\lambda)\vert Y
\ar[r]^{1_s} \ar[r] & \sh{O}_X(-\lambda)\vert Y \\
p^{-1}_{l+m+\Delta}\sh{O}_{l+m+\Delta} \ar[urrr]_{1_{p_{l+m+\Delta}}}}
\end{align}
for $p \in (s^{l+m+\Delta}_{l+m})^{-1}\sh{D}_{l+m \rightarrow l}$. Here, the
maps $1_{p_k}: p_k^{-1}\sh{O}_k \hookrightarrow \sh{O}_X(-\lambda)\vert Y$
are the monomorphisms corresponding to the first part of \eqref{eq:pullback}.\par
Let $v \in \mathfrak{g}$ and $m \geq 0$ be such that $[v,\mathfrak{n}^-_{l+m}] 
\subseteq \mathfrak{n}^-_l$ holds for all $l$. Then Lemma \autoref{Lemma:actionofsonnminusl} implies
$[sv,\mathfrak{n}^-_{l+m+\Delta}] \subseteq \mathfrak{n}^-_{l-\Delta}$. 
We recall from \eqref{eq:partiallplusml} that there
is a hence a unique element 
\begin{align*}
\partial^{l+m+\Delta}_{l-\Delta}(sv) \in \Ho^0(Y_{l+m+\Delta},F_1 \sh{D}_{l+m+\Delta \rightarrow l-\Delta})
\end{align*} such that $1_{p_{l+m+\Delta}} p_{l+m+\Delta}^{-1}(\partial^{l+m+\Delta}_{l-\Delta}(sv))=\partial(sv) 1_{p_{l-\Delta}}$ holds. 
We have the similar maps
\begin{align*}
^s \partial^{l+m}_l(v) \in \Ho^0(Y_{l+m+\Delta},(s^{l+m+\Delta}_{l+m})^{-1}F_1 \sh{D}_{l+m \rightarrow l})=\Ho^0(Y_{l+m},F_1\sh{D}_{l+m \rightarrow l})
\end{align*}
for $^s \partial$ instead of $\partial$. The identity 
\begin{align}\label{eq:relationbetweenepsilonands}
\epsilon(^s\partial^{l+m}_l(v)) = \partial^{l+m+\Delta}_{l-\Delta}(sv)
\end{align}
follows from \eqref{eq:partialands-1partial} and \eqref{eq:mapsbetweenOandmapsbetweens-1O} by 
uniqueness of $\partial^{l+m+\Delta}_{l-\Delta}(sv)$. Now \eqref{eq:identityinvolvingepsilon} 
and \eqref{eq:relationbetweenepsilonands} imply the statement of the theorem by definition of the $\mathfrak{g}$-action on $\Ho^j(X,\sh{M})$ and $\Ho^j(X,s_*\sh{M})$. 
\end{proof}

\section{Categories $\Hol_0^{\RVerma}(\lambda)$ and $\Hol_0^{\RVerma \Z}(\lambda)$}
\label{ssec:categoryHol0R(lambda)}

In this section we introduce the full subcategories $\Hol_0^{\RVerma}(\lambda)$
and $\Hol_0^{\RVerma \Z}(\lambda)$ of $\Hol(\lambda)$. They contain the category $\Hol_0(\lambda)$ of \cite{KT95}[p. 41]. 
We fix $i \in I$ in the whole section, e.g. in Definition \autoref{Def:Hol0R(lambda)} below.
Let $w \in W$ be such that $s_i w <w$. Associated to the embedding
$i_{w,l}$ \eqref{eq:iwl} we have 
a canonical morphism of functors $i_{w,l!} \rightarrow i_{w,l*}$
and in particular morphisms $\sh{R}_{!w}(\lambda,\mu)_l \rightarrow
\sh{R}_{*w}(\lambda,\mu)_l$. 
We define $\sh{R}_{!*w}(\lambda,\mu)_l \subseteq \sh{R}_{*w}(\lambda,\mu)_l$
as the image of this morphism.  
We have $\supp \sh{R}_{!*w}(\lambda,\mu)_l = \overline{X_w}$. 
As $i_{w,l}$ is an affine embedding of an irreducible subvariety and $\Omega^{(\mu)}_{\Gm}\boxtimes \Omega_{^{s_i}U_{s_iw}}$
is a simple $\sh{O}$-coherent $\sh{D}$-module it follows that $\sh{R}_{!*w}(\lambda,\mu)_l$ is the unique simple submodule of $\sh{R}_{*w}(\lambda,\mu)_l$
and simple quotient of $\sh{R}_{!w}(\lambda,\mu)_l$ by \cite{Ber}[p. 29]. 
Then $(\sh{R}_{!*w}(\lambda,\mu)_l)_{l \geq l_0}$ defines an object $\sh{R}_{!*w}(\lambda,\mu)
\in \Hol(\lambda)$. 
We have $\Dual \sh{R}_{!*w}(\lambda,\mu) \cong \sh{R}_{!*w}(\lambda,-\mu)$
as a direct consequence of \cite{Eic15}[Remark 2.3].  

\begin{Not}
When $\mu \in \Z$ we abbreviate the isomorphism class
$\sh{R}_{!*w}(\lambda)=\sh{R}_{!*w}(\lambda,\mu)$. 
\end{Not}

\begin{Rem}\label{Rem:LwlambdaequalsR!*wlambda}
By definition \cite{KT95}[p. 41] $\sh{L}_w(\lambda)$ is the unique simple submodule of $\sh{B}_w(\lambda)$. By \eqref{eq:seswithBwlambdaandRwlambda} below $\sh{B}_w(\lambda)$ is a submodule of $\sh{R}_{*w}(\lambda)$ and hence by uniqueness $\sh{L}_w(\lambda)=\sh{R}_{!*w}(\lambda)$ for $w$ such that $s_i w < w$. In the same way we find $\widetilde{s_i}_*\sh{L}_w(\lambda)=\sh{R}_{!*w}(\lambda)$
for $s_iw < w$. 
\end{Rem}

\begin{Def}\label{Def:Hol0R(lambda)}
We define $\Hol_0^{\RVerma}(\lambda)$ as the full
subcategory of $\Hol(\lambda)$ whose objects have a composition series
with simple quotients of the form $\sh{R}_{!*w}(\lambda,\mu)$ for some $\mu \in \C$
for $w$ such that $s_i w < w$ or to $\sh{L}_w(\lambda)$ 
or $\widetilde{s_i}_*\sh{L}_w(\lambda)$ for $s_i w > w$. We similarly define
the category $\Hol_0^{\RVerma \Z}(\lambda)$ requiring $\mu \in \Z$. 
\end{Def}

By Remark \autoref{Rem:LwlambdaequalsR!*wlambda} $\Hol_0^{\RVerma \Z}(\lambda)$ contains
$\Hol_0(\lambda)$. Analogous to the notion of an admissible closed subset of $X$ of
\cite{KT95}[p. 34] we have

\begin{Def}
We call a finite union of $s_i \overline{X_w}$ and
$\overline{X_w}$, $w \in W$, relaxed admissible closed subset of $X$. 
\end{Def}
We consider the collection of relaxed admissible closed subsets as a poset with respect to inclusion.
The collection of admissible closed subsets of $X$ is a cofinal subposet of it. 

\subsection{Properties of $\Hol_0^{\RVerma}(\lambda)$
and $\Hol_0^{\RVerma \Z}(\lambda)$}

\begin{Lemma}\label{Lemma:propertiesofHol0R(lambda)} \noindent
\begin{enumerate}
\item $\Dual$ preserves $\Hol_0^{\RVerma (\Z)}(\lambda)$. 
\item $\Hol_0^{\RVerma (\Z)}(\lambda)$ is a Serre subcategory of $\Hol(\lambda)$. 
\item The support of any object of $\Hol_0^{\RVerma}(\lambda)$ is a relaxed admissible
closed subset. 
\item Let $\sh{M} \in
\Hol_0^{\RVerma}(\lambda)$.  
Let $\sh{S}$ be a simple quotient of a composition series of $\sh{M}$. 
If $\sh{S}\cong\sh{L}_w(\lambda)$ or $\cong \sh{R}_{!*w}(\lambda,\mu)$, 
then we have $\overline{X_w} \subseteq \supp \sh{M}$. If $\sh{S} \cong \widetilde{s_i}_*\sh{L}_w(\lambda)$, 
then we have $s_i \overline{X_w} \subseteq \supp\sh{M}$. 
\end{enumerate}
\end{Lemma}

\begin{proof}
(1) This follows from the fact that $\Dual \sh{L}_w(\lambda) = \sh{L}_w(\lambda)$ and 
$\Dual \sh{R}_{!*w}(\lambda,\mu) = \sh{R}_{!*w}(\lambda,-\mu)$ holds and
from Lemma \autoref{Lemma:Holonomicdualitycommuteswithwidetildesi*}.\\
(2) This is shown by looking at the induced composition series of a subobject and a quotient object
of an object in $\Hol_0^{\RVerma (\Z)}(\lambda)$.\newline
(3) For a short exact sequence $0 \rightarrow \sh{M}_1 
\rightarrow \sh{M}_2 \rightarrow \sh{M}_3 \rightarrow 0$ in 
$\Hol_0^{\RVerma}(\lambda)$ we have $\supp \sh{M}_2 = \supp \sh{M}_1 \cup \supp\sh{M}_3$.
This implies the statement.\newline
(4) Assume there is a simple quotient $\sh{S} \cong \sh{L}_w(\lambda)$ in a composition
series of $\sh{M}$. Then $\supp \sh{M} \supseteq \supp \sh{L}_w(\lambda)
=\overline{X_w}$ holds. The other cases are completely similar. 
\end{proof}

\begin{Lemma}\noindent
\begin{enumerate}
\item $\widetilde{s_i}_* \sh{R}_{*w}(\lambda,\mu) \cong \sh{R}_{*w}(\lambda,-\mu)$
\item $\widetilde{s_i}_*$ preserves $\Hol_0^{\RVerma(\Z)}(\lambda)$. 
\end{enumerate}

\end{Lemma}
\begin{proof}
(1). It suffices to observe the isomorphism of $\sh{D}$-modules on $\Gm\times ^{s_i}U_{s_iw}$
\begin{align*}
\widetilde{s_i}^*\left(\Omega^{(\mu)}_{\Gm}\boxtimes \Omega_{^{s_i}U_{s_iw}}\right)
= \widetilde{s_i}^*\pi^* \Omega^{(\mu)}_{\Gm} = \pi^* J^* \Omega^{(\mu)}_{\Gm}
= \pi^* \Omega^{(-\mu)}_{\Gm} = \Omega^{(-\mu)}_{\Gm}\boxtimes \Omega_{^{s_i}U_{s_iw}}\;. 
\end{align*}
Here $\pi: \Gm \times ^{s_i}U_{s_iw} \rightarrow \Gm$ is the projection, 
$\widetilde{s_i}$ denotes the automorphism of $\Gm \times ^{s_i}U_{s_iw}$
induced by the same-named automorphism of $X$ and
$J$ is the automorphism of $\Gm$ given by $x \mapsto -x^{-1}$. The identity $\pi \widetilde{s_i} = J \pi$ is a direct computation. Furthermore $(\cdot)^*$ denotes the
inverse image of right $\sh{D}$-modules.\\
(2). (1) implies $\widetilde{s_i}_* \sh{R}_{!*w}(\lambda,\mu) \cong \sh{R}_{!*w}(\lambda,-\mu)$. 
Proposition \autoref{Prop:equivalencebetweenholonomiccats}(2) implies $(\widetilde{s_i}_*)^2\sh{B}_w(\lambda)
\cong \sh{B}_w(\lambda)$ as $\widetilde{s_i}^2X_w = s_i^2 X_w = X_w$, which in turn implies
$(\widetilde{s_i}_*)^2\sh{L}_w(\lambda) \cong \sh{L}_w(\lambda)$ for all $w \in W$. Lemma \autoref{Lemma:s*exactautoequivalenceofHol(lambda)}
and Definition \autoref{Def:Hol0R(lambda)} now imply the statement. 
\end{proof}

\subsection{Short exact sequences involving $\sh{R}_{*w}(\lambda)$}
Let $w \in W$ be such that $s_i w < w$. According to
\cite{Kas00}[Theorem 3.29 (2)] we have a distinguished triangle 
\begin{align}\label{eq:RGammaZdistinguishedtriangle}
\Rderived\Gamma_Z\sh{F} \rightarrow \sh{F} \rightarrow \Rderived\Gamma_{X^{\leq w}_l \setminus Z}\sh{F} \rightarrow
\end{align}
for any $\sh{F}$ in the bounded derived category of right $\sh{D}_{X^{\leq w}_l}$-modules, where we defined the closed subvariety $Z = \overline{X_{s_iw}}$
of $X^{\leq w}_l$. We set 
\begin{align}\label{eq:RGammaD}
\sh{F} = \Rderived \Gamma_{s_i X_w} \Omega_{X^{\leq w}_l}
\end{align}
and by \cite{Kas00}[Theorem 3.29 (1)] find the distinguished triangle
\begin{align*}%\label{eq:RGammadistinguishedtriangle}
\Rderived\Gamma_{X_{s_i w}}\Omega_{X^{\leq w}_l} \rightarrow \Rderived\Gamma_{s_i X_w} \Omega_{X^{\leq w}_l} \rightarrow \Rderived\Gamma_{s_i X_w \cap X_w}\Omega_{X^{\leq w}_l} \rightarrow
\end{align*}
as $s_iX_w \cap Z = X_{s_iw}$ and $s_i X_w\cap (X^{\leq w}_l\setminus Z) = s_iX_w \cap X_w$. 
The corresponding long exact sequence of cohomology reduces to
\begin{align*}
0 & \rightarrow \Ho^{\codim s_i X_w} \Rderived \Gamma_{s_i X_w}(\Omega_{X^{\leq w}_l})
\rightarrow \Ho^{\codim s_i X_w \cap X_w} \Rderived \Gamma_{X_w \cap s_i X_w}(\Omega_{X^{\leq w}_l})\\
& \rightarrow \Ho^{\codim X_{s_i w}} \Rderived \Gamma_{X_{s_iw}}(\Omega_{X^{\leq w}_l}) \rightarrow 0
\end{align*}
because the cohomology vanishes in all other degrees \cite{KT95}[p. 32]
as a consequence of the fact that the subvarieties $s_iX_w$, $s_iX_w \cap X_w$ and $X_{s_iw}$ are smooth, locally closed and affinely embedded in $X^{\leq w}_l$.
Thus we arrive at the exact sequence
\begin{align}\label{eq:seswithBwlambda-landRwlambdal}
0 \rightarrow \iota_{s_i X_w*}(\Omega_{s_i X_w})\otimes {\sh{O}_{X^{\leq w}_l}(\lambda)} \rightarrow \sh{R}_{*w}(\lambda)_l \rightarrow \sh{B}_{s_iw}(\lambda)_l \rightarrow 0\;,
\end{align}
where $\iota_{s_i X_w}: s_i X_w \hookrightarrow X^{\leq w}_l$. 
By Proposition \autoref{Prop:equivalencebetweenholonomiccats}(2) applied to the morphism $(\widetilde{s_i})^{l+\Delta}_l: X^{\leq w}_{l+\Delta} \rightarrow X^{\leq w}_l$ of Definition \autoref{Def:sllprime} we have
\begin{align*}
\iota_{s_i X_w*}(\Omega_{s_i X_w})\otimes {\sh{O}_{X^{\leq w}_l}(\lambda)}
=  (\widetilde{s_i}_* \sh{B}_w(\lambda))_l\;,
\end{align*}
see Definition \autoref{Def:s*Ml}. 
In this way we get the short exact sequence 
\begin{align}\label{eq:seswithBwlambda-andRwlambda}
0 \rightarrow \widetilde{s_i}_*\sh{B}_w(\lambda) \rightarrow \sh{R}_{*w}(\lambda)
\rightarrow \sh{B}_{s_iw}(\lambda) \rightarrow 0
\end{align}
in $\Hol(\lambda)$. 
Similarly, replacing $Z$ by $s_i Z$, which is still contained in $X^{\leq w}_l$,
in \eqref{eq:RGammaZdistinguishedtriangle} and $s_i X_w$ by $X_w$ in \eqref{eq:RGammaD} we obtain an
exact sequence
\begin{align}\label{eq:seswithBwlambdaandRwlambda}
0 \rightarrow \sh{B}_w(\lambda) \rightarrow \sh{R}_{*w}(\lambda) \rightarrow \widetilde{s_i}_*\sh{B}_{s_iw}(\lambda)
\rightarrow 0\;. 
\end{align}

The next lemma is a direct consequence of \eqref{eq:seswithBwlambda-andRwlambda} and Lemma
\autoref{Lemma:propertiesofHol0R(lambda)}(1) and (2). 
\begin{Lemma}
For $? \in \{*,!\}$ we have $\sh{R}_{?w}(\lambda) \in \Hol_0^{\RVerma \Z}(\lambda)$.
\qed
\end{Lemma}

\subsection{Properties of $\overline{\Ho^j}(X,\cdot)$ on $\Hol_0^{\RVerma}(\lambda)$}

Let us generalize \cite{KT95}[Proposition 3.3.1] and corollaries of it
from $\Hol_0(\lambda)$ to $\Hol_0^{\RVerma}(\lambda)$. 
\begin{Prop}\label{Prop:globalsectionsofMlisageneralizedweightmodule}
Let $\sh{M}=(\sh{M}_l)_{l \geq l_0}
\in \Hol_0^{\RVerma}(\lambda)$. 
Then $\Ho^j(Y_l,\sh{M}_l)$ is a generalized weight module with finite dimensional
generalized weight spaces and for $\mu \in \mathfrak{h}^*$ there is a $l^{\prime} \geq l_0$ such that
the natural map $\Ho^j(Y_{l_1},\sh{M}_{l_1})_{\mu} \rightarrow 
\Ho^j(Y_{l_2},\sh{M}_{l_2})_{\mu}$ is an isomorphism for
all $l_1 \geq l_2 \geq l^{\prime}$. 
\end{Prop}
\begin{proof}
We have $\supp \sh{M} = \bigcup_{j} \overline{X_{w_j}} \bigcup_k s_i \overline{X_{u_k}}$
for some finite collection $w_j, u_k \in W$. We can assume that $s_i \overline{X_{u_k}}$ is not
of the form $\overline{X_w}$. 
The proof is an induction on $\max_{j,k} \{\length(w_j),\length(u_k)\}$. 
The induction start is clear. Let 
\begin{align*}
0 \rightarrow \sh{M}^{(1)} \rightarrow \sh{M}^{(2)} \rightarrow \sh{M}^{(3)} \rightarrow 0
\end{align*}
be an exact sequence in $\Hol_0^{\RVerma}(\lambda)$. 
The long exact sequence
\begin{align*}
\dots & \rightarrow \Ho^{j-1}\left(Y_l,\sh{M}_l^{(3)}\right) \rightarrow \Ho^j\left(Y_l,\sh{M}_l^{(1)}\right) \rightarrow \Ho^j\left(Y_l,\sh{M}_l^{(2)}\right)
\rightarrow \Ho^j\left(Y_l,\sh{M}_l^{(3)}\right) \\
& \rightarrow \Ho^{j+1}\left(Y_l,\sh{M}_l^{(1)}\right) 
\rightarrow \dots
\end{align*}
of $\mathfrak{b}^-$-modules shows that if two of $\sh{M}^{(1),(2),(3)}$ satisfy the conclusion
of the proposition, then so does the third. As any object of $\Hol_0^{\RVerma}(\lambda)$
has finite length, it suffices to check the statement in the following cases: 
When $s_i w < w$ and $w \leq w_j$ for some $j$ for $\sh{M} = \sh{R}_{!*w}(\lambda,\mu)$.
When $s_i w > w$ and $w \leq w_j$ for some $j$ for $\sh{M}=\sh{L}_w(\lambda)$
and when $w \leq u_k$ for some $k$  for $\sh{M}=\widetilde{s_i}_*\sh{L}_w(\lambda)$. 
For $s_i w > w$ there is an exact sequence
\begin{align*}
0 \rightarrow \sh{L}_w(\lambda) \rightarrow \sh{B}_w(\lambda)
\rightarrow \sh{N} \rightarrow 0
\end{align*}
as a consequence of the definition of the
direct images.  As the second map in this sequence is an isomorphism on $X_w$
we have $\supp \sh{N} \subseteq \overline{X_w}\setminus X_w$. 
As $\overline{X_w}\setminus X_w = \bigcup_{j \in I:\; s_j w < w} \overline{X_{s_j w}}$
it follows that $\sh{N}$ satisfies the conclusion by induction hypothesis. 
Since $\sh{B}_w(\lambda)$ satisfies the conclusion by the
explicit expression for the $\mathfrak{h}$-module $\Ho^0(Y_l, \sh{B}_w(\lambda)_l)$ 
and we have $\Ho^j(Y_l,\sh{B}_w(\lambda)_l)=0$ for $j > 0$, see \cite{KT95}[Lemma 3.2.1], the statement
follows for $\sh{L}_w(\lambda)$.  The argument in case $\widetilde{s_i}_*\sh{L}_w(\lambda)$
is completely similar. 
For $s_i w < w$ there is again an exact sequence
\begin{align*}
0 \rightarrow \sh{R}_{!*w}(\lambda,\mu) 
\rightarrow \sh{R}_{*w}(\lambda,\mu) \rightarrow \sh{N}^{\prime} \rightarrow 0
\end{align*}
whose second map restricts to an isomorphism on $X_w \cap s_i X_w$. It follows
as above that $\supp \sh{N}^{\prime} \subseteq \overline{X_w} \setminus (X_w \cap s_i X_w)
= (\overline{X_w}\setminus X_w) \cup \overline{X_w}\setminus s_iX_w)$ and hence
the induction hypothesis shows that $\sh{N}^{\prime}$ satisfies the conclusion. 
Also, we have by Theorem \autoref{Thm:H0hmod}(1) an explicit expression for the $\mathfrak{h}$-module $\Ho^0(Y_l,\sh{R}_{*w}(\lambda,\mu)_l)$ and $\Ho^j(Y_l,\sh{R}_{*w}(\lambda,\mu)_l)=0$ 
for $j > 0$ was stated in Lemma \autoref{Lemma:cohvanishingforR*wlambdaalpha}. Thus the statement
also holds for $\sh{R}_{!*w}(\lambda,\mu)$.
\end{proof}

As a direct application of the definition of $\overline{\Ho^j}(X,\cdot)$, see \autoref{ssec:cohomologygroups}, we have
\begin{Cor}\label{Cor:lesforHol0R(lambda)}
A short exact sequence in 
$\Hol_0^{\RVerma}(\lambda)$ induces a long exact sequence of 
the $\mathfrak{g}$-modules $\overline{\Ho^j}(X,\cdot)$, $j \geq 0$.\qed
\end{Cor}

We can now deduce more about the cohomology of $\sh{R}_{!w}(\lambda)$. 
\begin{Cor}\label{Cor:cohvanishingforR!w(lambda)}\noindent
\begin{enumerate}
\item $\overline{\Ho^j}(X,\sh{R}_{!w}(\lambda))=0$ for $j > 0$ and $\lambda$ 
$\rho$-regular antidominant, i.e. $(\lambda+\rho)(h_i) < 0$ for all $i \in I$
\item (cf. \cite{KT95}[Lemma 3.3.6]) $\sum_{j=0}^{\infty}(-1)^j\charv \overline{\Ho^j}(X,\sh{M})
= \sum_{j=0}^{\infty}(-1)^j\charv \overline{\Ho^j}(X,\Dual\sh{M})$ holds for
$\sh{M} \in \Hol_0^{\RVerma \Z}(\lambda)$, where $\charv(\cdot)$ denotes the formal character
of the $\mathfrak{h}$-module. 
\item $\charv \overline{\Ho^0}(X,\sh{R}_{!w}(\lambda))=\charv\overline{\Ho^0}(X,\sh{R}_{*w}(\lambda))$
for $\lambda$ $\rho$-regular antidominant
\end{enumerate}
\end{Cor}
\begin{proof}
(1). This follows from Corollary \autoref{Cor:lesforHol0R(lambda)} applied to the
short exact sequence
\begin{align*}
0 \rightarrow \sh{M}_{s_iw}(\lambda) \rightarrow \sh{R}_{!w}(\lambda) \rightarrow \widetilde{s_i}_*\sh{M}_w(\lambda)
\rightarrow 0
\end{align*}
in $\Hol^{\RVerma \Z}_0(\lambda)$ dual to \eqref{eq:seswithBwlambda-andRwlambda}, \cite{KT95}[Theorem 3.4.1(i)] and 
Theorem \autoref{Thm:Hojofs_*M}.\\

(2). The property is stable under extension and hence it suffices to prove
it for $\sh{M}=\sh{L}_w(\lambda)$ and $\sh{M}=\widetilde{s_i}_*\sh{L}_w(\lambda)$, but
in that case $\Dual \sh{M} \cong \sh{M}$ as a consequence of Lemma \autoref{Lemma:Holonomicdualitycommuteswithwidetildesi*}.\\

(3) directly follows from (1), (2) and Lemma \autoref{Lemma:cohvanishingforR*wlambdaalpha}. 
\end{proof}

\section{Identification of $\overline{\Ho^j}(X,\sh{R}_{?w}(\lambda,\mu))$ as $\mathfrak{g}$-module in some cases}\label{sec:identificationofHjR!*wlambdaalphaasgmoduleinsomecases}

\subsection{Case $w=s_i$}\label{ssec:thecasew=si}
In this section we identify the cohomology groups 
\begin{align*}
\overline{\Ho^j}(X,\sh{R}_{?s_i}(\lambda,\mu))
\end{align*}
as $\mathfrak{g}$-modules for all values of $\lambda$, $\mu$, and $? \in \{*,!\}$. 
The method is similar to the identification $\overline{\Ho^0}(X,\sh{B}_1(\lambda))
\cong \MVerma(\lambda)$ of the global sections of the skyscraper $\sh{D}$-module
$\sh{B}_1(\lambda)$ supported at the distinguished point $1 \in X$. 
We only replace $1$ by the Schubert curve
$\overline{X_{s_i}}=X_{s_i}\sqcup X_1$ in $X$. We have isomorphisms
\begin{align}\label{eq:Schubertcurve}
\Proj^1 \xleftarrow{\cong} \SL_2/B^{\SL_2} \xrightarrow{\cong} G_i^{\prime}/G_i^{\prime >0}
\xrightarrow{\cong} G_i/G_i^{>0} \xrightarrow{\cong} \overline{X_{s_i}}\;. 
\end{align}
Here $B^{\SL_2}$ is the standard, upper triangular, Borel subgroup of $\SL_2$. $G_i^{>0}$ is the
subgroup of $G_i$, see Section \autoref{ssec:groupschemesTBPiGi}, with Lie algebra $\mathfrak{h}\oplus \C e_i$. $G_i^{\prime}$ is the derived
group of $G_i$ and $G_i^{\prime > 0}$ is the Borel subgroup of $G_i^{\prime}$ with Lie algebra $\C h_i \oplus \C e_i$. 
The restriction of
the group homomorphism $\SL_2 \rightarrow G_i^{\prime}$, which induces the second isomorphism, to the standard torus of $\SL_2$, identified with $\Gm$, is given by 
the first factor in the factorization of the coroot $h_i: \Gm \rightarrow T_i^{\prime} \hookrightarrow T$.
The fourth isomorphism is induced by
the action of $G_i$ on $1$, i.e. by the map $G_i \rightarrow X$, $g \mapsto g1$. 
The embedding $i_{s_i,l}$ of \eqref{eq:iwl} factors
into the composition of an open and a closed embedding 
\begin{align*}
\Gm \xrightarrow{j} \overline{X_{s_i}}=X_{s_i}\sqcup X_1 \xrightarrow{i_{s_i,l}^{\prime}} X^{\leq s_i}_l\;. 
\end{align*}
\begin{Rem}
We have $i_{s_i,l}^{\prime *}\sh{O}_{X^{\leq s_i}_l}(\lambda) \cong \sh{O}_{\Proj^1}(-\lambda(h_i))$.
This follows directly from the definition of $\sh{O}_{X^{\leq s_i}_l}(\lambda)$ and \eqref{eq:Schubertcurve}. 
\end{Rem}
We note that $ U^-_{s_i} = ^{s_i}U^-_{s_i} = U^-(\Phi^{<0}\setminus\{-\alpha_i\})$. 
The isomorphisms \eqref{eq:parametrizationofNXw} and \eqref{eq:parametrizationofNXsw} 
are equivariant for the group $T \ltimes U^-(\Phi^{<0}\setminus\{-\alpha_i\})$.
Let us denote $U^-(\Phi^{<0}\setminus\{-\alpha_i\})_l = U_l^- \backslash U^-(\Phi^{<0}\setminus\{-\alpha_i\})$. Thus we have an isomorphism of $T \ltimes U^-(\Phi^{<0}\setminus\{-\alpha_i\})_l$-schemes
\begin{align*}
X^{\leq s_i}_l &=N(X_{s_i})_l\cup N(X_1)_l  \cong U^-(\Phi^{<0}\setminus\{-\alpha_i\})_l \times (U(\alpha_i)\cup U^-(-\alpha_i))\\
&\cong U^-(\Phi^{<0}\setminus\{-\alpha_i\})_l \times \Proj^1
\end{align*}
and $\overline{X_{s_i}}$ identifies with $1\times \Proj^1$ under this
isomorphism. 

Let $\sh{M}$ be a $\sh{O}_{\Proj^1}(-\lambda(h_i))$-twisted holonomic right $\sh{D}$-module  on $\Proj^1$. 

\begin{Lemma}\label{Lemma:HojofextensionfromSchubertcurve}
There is an isomorphism of $\mathfrak{h}\ltimes \mathfrak{n}^-_i/\mathfrak{n}^-_l$-modules
\begin{align*}
\Ho^j(X^{\leq s_i}_l, i^{\prime}_{s_i,l *}\sh{M}) \cong \Uea(\mathfrak{n}^-_i/\mathfrak{n}^-_l)
\otimes_{\C}\Ho^j(\Proj^1, \sh{M})
\end{align*}
for $j \geq 0$. 
\end{Lemma}
\begin{proof}
We have $i^{\prime}_{s_i,l*}\sh{M} \cong \iota_*\C \boxtimes \sh{M}$ (external 
tensor product of $\sh{O}$-modules), where
$\iota: \{1\} \hookrightarrow U^-(\Phi^{<0}\setminus\{-\alpha_i\})_l$,
see e.g. \cite{Kas00}[p. 80]. Then we apply the Kuenneth formula
\cite{Kem93}[Proposition 9.2.4] to $\iota_*\C \boxtimes \sh{M}$ and use the fact that 
\begin{align*}
\Ho^j(U^-(\Phi^{<0}\setminus\{-\alpha_i\})_l,\iota_*\C)\cong \begin{cases} \Uea\mathfrak(\mathfrak{n}^-_i/\mathfrak{n}^-_l) & j=0 \\
0 & j > 0 \end{cases}
\end{align*}
as $\mathfrak{h}\ltimes \mathfrak{n}^-_i/\mathfrak{n}^-_l$-module. 
\end{proof}

By \eqref{eq:Schubertcurve} we see that the
subalgebra $\mathfrak{g}_i$ of $\mathfrak{g}$ 
acts on the subspace $\C1 \otimes \Ho^j(\Proj^1,\sh{M})$
of $\Ho^j(X^{\leq s_i}_l,i^{\prime}_{s_i,l*}\sh{M})$. Since $\sh{M}$ is holonomic, 
$i^{\prime}_{s_i *}\sh{M}= (i^{\prime}_{s_i,l*}\sh{M})_{l \geq l_0}$ is an object of $\Hol(\lambda)$. 

\begin{Thm}\label{Thm:HojofdirectimageofSchubertcurve}
 $\overline{\Ho^j}(X,i^{\prime}_{s_i*}\sh{M}) \cong \Uea\mathfrak{g}\otimes_{\Uea\mathfrak{p}_i}
\Ho^j(\Proj^1,\sh{M})$ as $\mathfrak{g}$-module, where on the rhs $\mathfrak{p}_i$ acts via \eqref{eq:projfrompitogi}. In particular $\overline{\Ho^j}(X,i^{\prime}_{s_i*}\sh{M})=0$
for $j \geq 2$. 
\end{Thm}
\begin{proof}
Lemma \autoref{Lemma:HojofextensionfromSchubertcurve} induces an isomorphism
\begin{align}\label{eq:Hojis_i*M}
\overline{\Ho^j}(X,i^{\prime}_{s_i*}\sh{M})\cong \Uea \mathfrak{n}_i^-\otimes_{\C}
\Ho^j(\Proj^1,\sh{M})
\end{align}
of $\mathfrak{h}\ltimes \mathfrak{n}_i^-$-modules. The embedding $M=\Ho^j(\Proj^1,\sh{M}) \hookrightarrow N=\overline{\Ho^j}(X,i^{\prime}_{s_i*}\sh{M})$ is a map of $\mathfrak{p}_i$-modules. 
Via the adjunction $\Hom_{\mathfrak{g}}(\Uea\mathfrak{g}\otimes_{\Uea\mathfrak{p_i}}M,N)
\cong \Hom_{\mathfrak{p}_i}(M,N)$ we obtain a $\mathfrak{g}$-map $\phi: \Uea\mathfrak{g}\otimes_{\Uea\mathfrak{p}_i}M \rightarrow N$. 
\eqref{eq:Hojis_i*M} implies that $\phi$ is surjective and injective.
\end{proof}

\cite{Eic15}[Theorem 5.2] implies
\begin{Cor}
Let $\mu \in \C\setminus \Z$. We have $\overline{\Ho^0}(X,\sh{R}_{*s_i}(\lambda,\mu))\cong \RVerma(\lambda,\mu)$ as $\mathfrak{g}$-module. 
\end{Cor}

\cite{Eic15}[Theorem 5.3 and 5.4] imply
\begin{Cor}
Let $-\lambda(h_i) \geq 2$. We have isomorphisms of $\mathfrak{g}$-modules
$\overline{\Ho^0}(X,\sh{R}_{*s_i}(\lambda))\cong \RVerma(\lambda,=)^{\veesl2}$
and $\overline{\Ho^0}(X,\sh{R}_{!s_i}(\lambda)) \cong \RVerma(\lambda,=)$,
$\overline{\Ho^1}(X,\sh{R}_{!s_i}(\lambda))=0$. 
\end{Cor}

In the same way \cite{Eic15}[Theorem 5.1-5.4] yield an explicit description of
the $\mathfrak{g}$-modules 
$\overline{\Ho^j}(X,\sh{B}_{s_i}(\lambda))$, $\overline{\Ho^j}(X,\sh{M}_{s_i}(\lambda))$, $\overline{\Ho^j}(X,\sh{R}_{*s_i}(\lambda,\mu))$ 
and $\overline{\Ho^j}(X,\sh{R}_{!s_i}(\lambda,\mu))$ as induced modules for $j =0,1$ and all values
of $\lambda$ and $\mu$. Rather than listing every case here, we comment on these results. We recall that the condition that $\lambda$ be
$\rho$-regular antidominant implies $-\lambda(h_i)\geq 2$.
The vanishing of the
$\overline{\Ho^1}(X,\cdot)$ obtained in the above way 
is consistent with \cite{KT95}[Theorem 3.4.1(i)] and in fact these $\overline{\Ho^1}(X,\cdot)$ vanish for 
any $\lambda$ such that $-\lambda(h_i) \geq 2$.
The nonvanishing $\overline{\Ho^1}(X,\cdot)$ are parabolic Verma modules 
(this notion was recalled in \autoref{ssec:commentsonlit}). 
In particular, by \cite{Eic15}[Theorem 5.1], we have $\overline{\Ho^1}(X,\sh{M}_{s_i}(\lambda)) \neq 0$ if and only if $-\lambda(h_i) \leq 0$.\par
Using Theorem \cite{Eic15}[Theorem 5.5] one can produce
further induced $\mathfrak{g}$-modules as $\overline{\Ho^0}(X,\cdot)$ of
objects of $\Hol(\lambda)$ similar to $\sh{R}_{?s_i}(\lambda)$, $? \in \{*,!\}$.
These $\mathfrak{g}$-modules are not isomorphic to the ones above, but have the same structure as $\mathfrak{h}$-module.

\subsection{Case $\sh{R}_{*w}(\lambda)$}
In this section we identify $\overline{\Ho^0}(X,\sh{R}_{*w}(\lambda))$
as a $\mathfrak{g}$-module for all $w \in W$. The case $w=s_i$ has already 
been covered in \autoref{ssec:thecasew=si}. 

\begin{Thm}\label{Thm:Ho0R*wlambdaasgmod}
Let $\lambda$ be $\rho$-regular antidominant, i.e. 
$(\lambda+\rho)(h_i) < 0$ for all $i \in I$. Then $\overline{\Ho^0}(X,\sh{R}_{*w}(\lambda))
\cong \RVerma(w\cdot \lambda,=)^{\vee}$ as $\mathfrak{g}$-module. 
\end{Thm}
We will prove this theorem in \autoref{ssec:pfofthm} after some preparations. 
The strategy is similar to the highest weight case, namely we adapt the arguments of
\cite{KT95} identifying the $\mathfrak{g}$-module $\overline{\Ho^0}(X,\sh{B}_w(\lambda))$
as the dual Verma module $\MVerma(w\cdot \lambda)^{\vee}$ (this statement itself will
also be an ingredient in our proof).  

\subsubsection{Algebraic statements about $\RVerma(\lambda,=)$}
\emph{In this section we require $\lambda(h_i) \geq 0$.}

\begin{Rem}\label{Rem:basisinRVerma=}
Let us fix a basis $(X_k)_{k \in \Z_{>0}}$ of root vectors
of $\mathfrak{n}^-_i$. 
$\RVerma(\lambda,=)$ is freely generated by the action of $\mathfrak{n}^-_i$
from the ``top level'' subspace $\C_{\lambda}\otimes \RVerma^{\mathfrak{sl}_2} (\lambda(h_i),=)$
of $\RVerma(\lambda,=)$. Thus it
has a basis $X^{n_1}_{k_1}\dots X^{n_N}_{k_N}v$, where $N \in \Z_{\geq 0}$, $n_j \in \Z_{>0}$ 
and $k_1 > \dots > k_N$ and $v$ runs through a basis of the ``top level''.
\end{Rem} 

We will refer to $X^{n_1}_{k_1}\dots X^{n_N}_{k_N} \in \Uea \mathfrak{n}^-_i$ as ordered monomial and abbreviate it as $X$. 

\begin{Lemma}\label{Lemma:Rlambda=hasnononzergiprimefinitevectors}
$\RVerma(\lambda,=)$ does not have nonzero $\mathfrak{g}_i^{\prime}$-finite vectors.
\end{Lemma}
\begin{proof}
By construction of $\RVerma^{\mathfrak{sl}_2}(\lambda(h_i),=)$ in \cite{Eic15} there is a weight vector $v$ in the ``top level'' $\C_{\lambda}\otimes \RVerma^{\mathfrak{sl}_2}(\lambda(h_i),=)$
of $\RVerma(\lambda,=)$ such that the (repeated) action of $e_i$ and $f_i$ on $v$ spans the whole ``top level''.
In the argument below we will consider the collection of vectors 
\begin{align}\label{eq:basisinRVermafromgenv}
X e_i^n v\;,\; n > 0\;,\;  X f_i^nv\;,\; n \geq 0\;,
\end{align}
where $X$ runs through the ordered monomials. By Remark \autoref{Rem:basisinRVerma=} this
is a basis of $\RVerma(\lambda,=)$. 
We claim that by letting
$e_i^m$ or $f_i^m$ act on the vector $w=Xv$, where $X$ is an ordered monomial and bringing it to the right, noting $[e_i,\mathfrak{n}^-_i] \subseteq \mathfrak{n}^-_i$
and $[f_i,\mathfrak{n}^-_i] \subseteq \mathfrak{n}^-_i$, we find that the subspace $\C[e_i]w$ and
$\C[f_i]w$ is infinite dimensional. Indeed, 
\begin{align*}
e_i Xv
= X e_i v +\text{``brackets of $e_i$ with $X_{k_j}$''}\;. 
\end{align*}
But these brackets are in $\mathfrak{n}_i^-$ and hence are a linear combination of $X^{\prime}v$, where $X^{\prime}$ is an ordered monomial
(to make the monomials ordered one has to form brackets in $\mathfrak{n}^-_i$).
By induction on $m$ it follows that
the element $Xe_i^m v$ of the basis \eqref{eq:basisinRVermafromgenv} occurs in $e_i^m w$ 
with nonzero coefficient. 
Hence $\C[e_i]w$ is an infinite dimensional subspace. 
The argument for $\C[f_i] w$ is completely similar. 
Given now a general vector $w \in \RVerma(\lambda,=)$, we can expand it in the basis \eqref{eq:basisinRVermafromgenv} and hence write it uniquely in the form
\begin{align*}
w = \sum_{l=1}^{L} A_l f_i^{a_l} v
+ \sum_{p=1}^{P} B_pe_i^{b_p} v\;,
\end{align*}
where $A_l$ and $B_p$ are nonzero linear combinations of ordered monomials and $0 \leq a_1 < a_2 < \dots$ and $0 < b_1 < b_2 < \dots$. 
If $L \geq 1$, then the above argument shows that each basis element of \eqref{eq:basisinRVermafromgenv} occurring in $A_L f_i^{a_L+m}v$
with nonzero coefficient will also occur in $f_i^m w$ with nonzero coefficient. 
Hence $\C[f_i]w$ is infinite dimensional. 
Similarly we show in case $P \geq 1$ that $\C[e_i]w$ is infinite dimensional. 
Thus $w$ is not $\mathfrak{g}_i^{\prime}$-finite unless $L=P=0$, in which case $w=0$. 
\end{proof}

We note that from Remark \autoref{Rem:basisinRVerma=}
it follows that $\MVerma(s_i\cdot \lambda)\oplus \MVerma(s_i\cdot \lambda)^{\widetilde{s_i}}$ can
be considered as a submodule of $\RVerma(\lambda,=)$ in a unique way.
The ``extremal'' weight of $\MVerma(s_i\cdot\lambda)^{\widetilde{s_i}}$ is
$s_i(s_i\cdot \lambda)=\lambda+\alpha_i$.
The following proposition can be regarded as 
a version of \cite{KT95}[Proposition 1.2.2(ii)],
where $\MVerma(\lambda)$ is replaced by $\RVerma(\lambda,=)$ 
and the submodule $\MVerma(s_i \cdot \lambda)$ by 
the submodule $\MVerma(s_i\cdot \lambda)\oplus \MVerma(s_i\cdot \lambda)^{\widetilde{s_i}}$. 

\begin{Prop}\label{Lemma:nonzerosubofRwdotlabmda=intersects}
Any nonzero submodule of $\RVerma(\lambda,=)$ intersects
$\MVerma(s_i\cdot \lambda)\oplus \MVerma(s_i \cdot \lambda)^{\widetilde{s_i}}$ nontrivially. 
\end{Prop}

\begin{proof}
We consider the quotient $\mathfrak{g}$-module $W= \RVerma(\lambda,=)/\MVerma(s_i\cdot \lambda)\oplus \MVerma(s_i\cdot \lambda)^{\widetilde{s_i}}$.
Let $V$ be the image of the ``top level'' of $\RVerma(\lambda,=)$ in $W$. It is a finite dimensional $\mathfrak{g}_i$-module. 
Then $W$ is spanned by the vectors $w= X v$, where $X$ runs through the ordered monomials and $v$ through a basis of weight vectors of $V$. The weights of $W$ are in particular contained in
$\wt(\Uea\mathfrak{n}_i^-)+\wt(V)$. On the other hand we have $\wt((\Uea\mathfrak{g_i})w) \subseteq \Z\alpha_i + \wt(w)$.  
Thus it follows from \cite{Kas90}[Lemma 5.1.6] that $\wt((\Uea\mathfrak{g_i})w)$ is a finite set. But then it is easy to conclude that $(\Uea\mathfrak{g}_i)w$ is finite dimensional. 
It follows that any (not necessarily weight) vector in $W$ is $\mathfrak{g}_i$-finite. Now the statement follows from Lemma \autoref{Lemma:Rlambda=hasnononzergiprimefinitevectors}. 
\end{proof}

\subsubsection{Proof of the theorem}\label{ssec:pfofthm}
We recall that the $\mathfrak{h}$-module structure of $\overline{\Ho^0}(X,\sh{R}_{*w}(\lambda))$
is known from Theorem \autoref{Thm:H0hmod}(2) and coincides with the one of $\overline{\Ho^0}(X,\sh{R}_{*w}(\lambda))^{\vee}$. This will be used without mention multiple times in the subsequent arguments. We first note that we have a long exact sequence for $\Ho^j\left(X^{\leq w}_l,\cdot\right)$
associated to the short exact sequence \eqref{eq:seswithBwlambda-landRwlambdal}. 
It follows that we have a short exact sequence 
\begin{align*}
0 \rightarrow \overline{\Ho^0}(X,\widetilde{s_i}_*\sh{B}_w(\lambda))
\rightarrow \overline{\Ho^0}(X,\sh{R}_{*w}(\lambda)) \rightarrow \overline{\Ho^0}(X,\sh{B}_{s_iw}(\lambda))
\rightarrow 0 
\end{align*}
of $\mathfrak{g}$-modules. 
Here we applied Theorem \autoref{Thm:Hojofs_*M}, which shows $\overline{\Ho^1}(X,\widetilde{s_i}_*\sh{B}_w(\lambda))
\cong \overline{\Ho^1}(X,\sh{B}_w(\lambda))^{\widetilde{s_i}}=0$. 
(Corollary \autoref{Cor:lesforHol0R(lambda)} on the other hand is not needed.)
Using the identification $\overline{\Ho^0}(X,\sh{B}_{s_iw}(\lambda))
\cong \MVerma(s_iw \cdot \lambda)^{\vee}$ proven in \cite{KT95}[Theorem 3.4.1(ii)]
for $\lambda$ $\rho$-regular antidominant and applying $(\cdot)^{\vee}$ to the above surjection $\overline{\Ho^0}(X,\sh{R}_{*w}(\lambda)) \rightarrow \overline{\Ho^0}(X,\sh{B}_{s_iw}(\lambda))$
we obtain an injection $\psi: \MVerma(s_iw\cdot \lambda) \hookrightarrow 
\overline{\Ho^0}(X,\sh{R}_{*w}(\lambda))^{\vee}$. Similarly, we obtain from the short
exact sequence \eqref{eq:seswithBwlambdaandRwlambda} an injection
$\psi^{\widetilde{s_i}}: \MVerma(s_iw\cdot \lambda)^{\widetilde{s_i}} \hookrightarrow \overline{\Ho^0}(X,\sh{R}_{*w}(\lambda))^{\vee}$. 
\begin{Rem}\label{Rem:uniquenessofpsiandpsisi}
The morphisms $\psi$ and $\psi^{\widetilde{s_i}}$ are the unique, up to scalar, morphisms with given source and target by the universal property of the Verma
module and the fact that the corresponding weight spaces of $\overline{\Ho^0}(X,\sh{R}_{*w}(\lambda))^{\vee}$
are one dimensional. 
\end{Rem}
Let us now describe the ``top level'' of $\overline{\Ho^0}(X,\sh{R}_{*w}(\lambda))^{\vee}$,
i.e. the subspace 
\begin{align*}
L=\bigoplus_{\mu \in \Z\alpha_i+w\cdot\lambda}\overline{\Ho^0}(X,\sh{R}_{*w}(\lambda))^{\vee}_{\mu}\;. 
\end{align*} 
This is a $\mathfrak{g}_i$-submodule of $\overline{\Ho^0}(X,\sh{R}_{*w}(\lambda))^{\vee}$. 
\begin{Lemma}
$L \cong \C_{w\cdot \lambda}\otimes \RVerma^{\mathfrak{sl}_2}((w\cdot\lambda)(h_i),=)$
as $\mathfrak{g}_i$-module, where on the rhs $\mathfrak{g}_i$ acts according to the direct sum decomposition
of \eqref{eq:projfrompitogi} . 
\end{Lemma}
\begin{proof}
We can consider $L^{\prime}=\bigoplus_{\mu \in \Z\alpha_i+w\cdot\lambda}\overline{\Ho^0}(X,\sh{R}_{*w}(\lambda))_{\mu}$
as a subspace of $\Ho^0(N(X_{s_iw})_l,\sh{R}_{*w}(\lambda)_l)$.
We recall \eqref{eq:R*wrestrtoNXswisBw}. There is an explicit description of the action of $f_i$ by
the first order differential operator $\partial_z$ on $L^{\prime} \subseteq \Ho^0(N(X_{s_iw})_l,\sh{B}_w(\lambda)_l)$, where $z$ is the affine coordinate appearing in \eqref{eq:transfunction}.
Indeed, such a description is given in \cite{KT95}[near (3.6.10)] for $\sh{M}_w(\lambda)_l$ 
instead of $\sh{B}_w(\lambda)_l$ and renders to the case of $\sh{B}_w(\lambda)_l$. It shows
\begin{align*}
f_i \overline{\Ho^0}(X,\sh{R}_{*w}(\lambda))_{s_i w\cdot \lambda + \alpha_i}=0\qquad f_i \overline{\Ho^0}(X,\sh{R}_{*w}(\lambda))_{\mu} \neq 0\quad \mu \in (\Z\alpha_i + w\cdot\lambda)
\setminus\{s_iw\cdot \lambda+\alpha_i\}\;. 
\end{align*}
Similarly, $\sh{R}_{*w}(\lambda)_l \vert N(X_w)_l \cong \sh{B}_w(\lambda)_l \vert N(X_w)_l$
and we can consider $L^{\prime}$ as a subspace of $\Ho^0(N(X_w)_l,\sh{R}_{*w}(\lambda)_l)$.
The action of $e_i$ is given by a $\C$-linear map $\Ho^0(N(X_w)_{l+1},\sh{B}_w(\lambda)_{l+1}) \rightarrow
\Ho^0(N(X_w)_l,\sh{B}_w(\lambda)_l)$, which on $L^{\prime}$ can be described as the first order differential operator $\partial_x$, where $x$ is the affine coordinate on $U(\alpha_i)$, see \cite{KT95}[p. 50]. 
This shows
\begin{align*}
e_i \overline{\Ho^0}(X,\sh{R}_{*w}(\lambda))_{w\cdot \lambda} = 0\qquad e_i \overline{\Ho^0}(X,\sh{R}_{*w}(\lambda))_{\mu} \neq 0\quad \mu \in (\Z\alpha_i+w\cdot\lambda)
\setminus\{ w\cdot \lambda\}\;. 
\end{align*}
The statement then follows by applying the definition of $(\cdot)^{\vee}$. 
\end{proof}

The isomorphism of this lemma induces an injection of $\mathfrak{p}_i$-modules 
\begin{align*}
\C_{w\cdot \lambda}\otimes \RVerma^{\mathfrak{sl}_2}((w\cdot\lambda)(h_i),=)
\hookrightarrow \overline{\Ho^0}(X,\sh{R}_{*w}(\lambda))^{\vee}\;,
\end{align*} 
where $\mathfrak{p}_i$ acts via \eqref{eq:projfrompitogi} on the source. 
By Definition \autoref{Def:RVermalambdalapha} of $\RVerma(w\cdot \lambda,=)$ we obtain a $\mathfrak{g}$-map
$\phi: \RVerma(w\cdot \lambda,=) \rightarrow \overline{\Ho^0}(X,\sh{R}_{*w}(\lambda))^{\vee}$. 
The restrictions $\phi \vert \MVerma(s_iw \cdot \lambda)$ and $\phi \vert \MVerma(s_iw\cdot \lambda)^{\widetilde{s_i}}$ are both nonzero by construction and hence by
Remark \autoref{Rem:uniquenessofpsiandpsisi} a nonzero multiple of $\psi$ and $\psi^{\widetilde{s_i}}$,
respectively. Hence $\ker \phi$ does not intersect $\MVerma(s_iw\cdot \lambda)\oplus \MVerma(s_iw\cdot \lambda)^{\widetilde{s_i}}$ and we infer $\ker \phi = 0$ by Proposition \autoref{Lemma:nonzerosubofRwdotlabmda=intersects}. Thus $\phi$ is an injection and hence an isomorphism. This concludes the proof of Theorem \autoref{Thm:Ho0R*wlambdaasgmod}.

\bibliographystyle{alpha}
\bibliography{references}

\end{document}

%% file: preamble.tex
\usepackage{amsmath, amsthm, amssymb}

\usepackage{tensor}
\usepackage{bbm}%lowercase mathbb
\usepackage{geometry}
\usepackage{soul}%underline 
\usepackage[pdftex]{graphicx}
\usepackage{float}
\usepackage{subfig}
\usepackage[english]{babel}
\usepackage{tikz} % graphical language
\usepackage{setspace}
\usetikzlibrary{arrows,decorations.pathreplacing,decorations.markings,shapes.geometric}
\tikzset{
    bt/.style={draw=blue,thick},
    ns/.style={circle,draw=blue,fill=blue, inner sep=0pt, minimum size=2mm},
    string/.style={draw=#1, postaction={decorate}, decoration={markings,mark=at position .45 with {\arrow[blue]{triangle 60}}}},
    doublestring/.style={draw=#1, postaction={decorate}, decoration={markings, mark=at position .7 with {\arrow[blue]{triangle 60}}, 
    mark=at position .3 with {\arrowreversed[blue]{triangle 60}}}},
    costring/.style={draw=#1, postaction={decorate}, decoration={markings,mark=at position .55 with {\arrow[draw=#1]{<}}}},
    arr/.style={string=blue, thick},
    doublearr/.style={doublestring=blue, thick},
    lin/.style={blue},
    dlin/.style = {blue, dashed, thick},
    dot/.style={circle,draw=#1,fill=#1,inner sep=1pt},
}

\usepackage{dsfont}
\usepackage{color}
\usepackage[color,matrix,arrow]{xy}%diagrams
\graphicspath{{graphics/}}
\numberwithin{equation}{section}
\usepackage{fancyhdr}
\theoremstyle{definition}

\newtheorem{Thm}{Theorem}[section]
\newtheorem{Cor}{Corollary}[section]
\newtheorem{Prop}{Proposition}[section]
\newtheorem{Lemma}{Lemma}[section]
\newtheorem{Rem}{Remark}[section]
\newtheorem{Def}{Definition}[section]

\newtheorem{Not}{Notation}[section]

\newcommand{\sh}[1]{\mathcal{#1}}

\DeclareMathOperator{\id}{id}

\DeclareMathOperator{\wt}{wt}

\DeclareMathOperator{\Hol}{Hol}

\DeclareMathOperator{\Uea}{\mathcal{U}}

\DeclareMathOperator{\Sym}{S}

\DeclareMathOperator{\SL}{SL}

\DeclareMathOperator{\C}{\mathbb{C}}
\DeclareMathOperator{\R}{\mathbb{R}}
\DeclareMathOperator{\Z}{\mathbb{Z}}

\DeclareMathOperator{\End}{End}

\DeclareMathOperator{\Lie}{Lie}

\DeclareMathOperator{\Proj}{\mathbb{P}}

\DeclareMathOperator{\Hom}{Hom}
\DeclareMathOperator{\length}{\ell}
\DeclareMathOperator{\codim}{codim}

\DeclareMathOperator{\Gm}{\mathbb{G}_m}

\DeclareMathOperator{\Spec}{Spec}

\DeclareMathOperator{\Ho}{H}

\DeclareMathOperator{\charv}{char}
\DeclareMathOperator{\supp}{supp}
\DeclareMathOperator{\Dual}{\mathbb{D}}

\DeclareMathOperator{\Aff}{\mathbb{A}}

\DeclareMathOperator{\MVerma}{M}

\DeclareMathOperator{\RVerma}{R}
\DeclareMathOperator{\Rderived}{R}
\DeclareMathOperator{\modulecat}{mod}

\DeclareMathOperator{\thin}{thin}

\DeclareMathOperator{\veesl2}{\vee \mathfrak{sl}_2}
\DeclareMathOperator{\Loop}{L}
\DeclareMathOperator{\SU}{SU}

\usepackage[pdftex,
bookmarks=true,
bookmarksnumbered=true,
hypertexnames=false,
colorlinks = true,%colorlinks = true if you want color
linkcolor = black,
citecolor = black,
urlcolor = black]{hyperref}

\title{Relaxed highest weight modules from $\sh{D}$-modules on the Kashiwara flag scheme}